\documentclass[11pt]{article}

\usepackage{setspace}
\usepackage[margin=1in]{geometry}
\usepackage{amsthm}
\usepackage{graphicx}
\usepackage{amsmath}
\usepackage{amssymb}
\usepackage{amsfonts}
\usepackage{mathtools}
\usepackage{hyperref}
\usepackage{tikz} 
\usepackage{mathrsfs}
\usepackage{enumerate}
\usepackage{csquotes}
\usepackage{subcaption}

\newcommand{\R}{\mathbb{R}}
\newcommand{\N}{\mathbb{N}}
\newcommand{\Z}{\mathbb{Z}}
\newcommand{\Q}{\mathbb{Q}}
\newcommand{\proj}{\operatorname{proj}}
\newcommand{\ext}{\operatorname{ext}}
\newcommand{\cl}{\operatorname{cl}}
\newcommand{\conv}{\operatorname{conv}}
\newcommand{\cone}{\operatorname{cone}}

\newcommand{\rec}{\operatorname{rec}}

\newcommand{\inter}{\operatorname{int}}

\newcommand{\lin}{\operatorname{lin}}

\newcommand{\CG}{\text{CG}}
\newcommand{\I}{\mathscr{I}}

\newtheorem{example}{Example}
\newtheorem{theorem}{Theorem}
\newtheorem{lemma}{Lemma}
\newtheorem{corollary}{Corollary}
\newtheorem{definition}{Definition}
\newtheorem{proposition}{Proposition}

\newcounter{claim} 
\newenvironment{claim}[1][]
{\refstepcounter{claim} \begin{trivlist} \item[] {\bf Claim~\theclaim.}\space#1 \itshape}
{\end{trivlist}}

\newenvironment{cpf}
{\begin{trivlist} \item[] {\em Proof of claim. }}
{$\hfill\diamond$ \end{trivlist}}

\title{On the Polyhedrality of the Chv\'atal-Gomory Closure}

\date{}

\begin{document}

\author{
Haoran Zhu
\thanks{Department of Industrial and Systems Engineering,
             University of Wisconsin-Madison.
             E-mail: {\tt hzhu94@wisc.edu}.
             }
}

\maketitle

\begin{abstract} 
In this paper, we provide an equivalent condition for the Chv\'{a}tal-Gomory (CG) closure of a closed convex set to be finitely-generated. Using this result, we are able to prove that, for any closed convex set that can be written as the Minkowski sum of a compact convex set and a closed convex cone, its CG closure is a rational polyhedron if and only if its recession cone is a rational polyhedral cone. As a consequence, this generalizes and unifies all the currently known results, for the case of rational polyhedron \cite{schrijver1980cutting} and compact convex set \cite{dadush2014chvatal}. 

\end{abstract}

\section{Introduction}
Cutting-plane method is one of the most fundamental techniques for solving (mixed) integer programming problems, and often times in practice, it is combined with the branch-and-bound method.
Since the early days of Integer Programming (IP), numerous types of cutting-planes have been introduced and studied in the literature, several of them have also been widely implemented into commercial solvers. 
Among those cuts, Chv\`{a}tal-Gomory (CG) cut (\cite{MR102437,CHVATAL1973305}) was the first cutting-plane that has ever been proposed, and various interesting results have been obtained from both the theoretical and practical point of view (see, e.g., \cite{MR986890,bockmayr1999chvatal,MR2306128}). 

One of the theoretical questions regarding to CG cut is, what are the structural properties of the region obtained from intersecting all of those cuts? In the terminology of cutting-plane theory, this region is referred to as \emph{Chv\`{a}tal-Gomory (CG) closure}. 
Although the definition of CG cut has traditionally been defined with respect to a rational polyhedron for an Integer Linear Programming (ILP) problem, they straightforwardly generalize to the nonlinear setting and hence can be used for convex Integer Nonlinear Programming (INLP). Let $K$ be a closed convex set and let $\sigma_K$ denote the support function of $K: \sigma_K(c) = \sup_{x \in K} c x$. For the ease of notation, here we abbreviate the inner product $c^T x$ as $cx$. Given $c \in \Z^n$, the CG cut for $K$ that is derived from $c$ is defined as:
$cx \leq \lfloor \sigma_K(c) \rfloor.$
Then, the so-called CG closure of $K$ is:
$$
K': = \bigcap_{c \in \Z^n} \{x \in \R^n \mid cx \leq \lfloor \sigma_K(c) \rfloor\}.
$$
Here the CG closure $K'$ is essentially obtained from the intersection of potentially infinitely many half-spaces, hence the polyhedrality of $K'$ is unclear. 
As named in \cite{MR2969261}, here we call $K'$ is \emph{finitely-generated}, if there exists a finite set $F \subseteq \Z^n$, such that
$K' = \bigcap_{f \in F} \{x \in \R^n \mid  f x \leq \lfloor \sigma_K(f) \rfloor\}.$
Obviously, a finitely-generated CG closure is a rational polyhedron.
Four decades ago, Schrijver \cite{schrijver1980cutting} shows that, when $K$ is a rational polyhedron, $K'$ is finitely-generated. Schrijver further asks the following question: when $K$ is an irrational polytope, is $K'$ still a (rational) polytope? 

As attempts to fully understand this question, a series of studies have been conducted for the polyhedrality of CG closure of various convex sets. In \cite{dey2010chvatal}, Dey and Vielma show that, the CG closure of a bounded full-dimensional ellipsoid, described by rational data, is a rational polytope. In \cite{dadush2011chvatal}, Dadush, Dey and Vielma show the CG closure of a set obtained as an intersection of a strictly convex body and a rational polyhedron is a polyhedron. Along this line of work, in \cite{dadush2014chvatal}, the same group of authors extend the same result to compact convex sets, therefore giving affirmation answer to the long-standing open problem raised by Schrijver. Almost simultaneously this problem was also proved by Dunkel and Schulz \cite{dunkel2013gomory} independently, where they specifically prove for the case of irrational polytope, instead of a more general compact convex set. All their proofs are very much involved, a few years later Braun and Pokutta \cite{braun2014short} give a short proof for the same result as \cite{dadush2014chvatal}. However,
no matter how different these proofs might seem, they all share some high-level similarities. For example, they all rely heavily on the \emph{homogeneity property} of CG closure: $F' = K' \cap F$ for any face $F$ of $K$. By inductive hypothesis that $F'$ is rational polyhedral and some additional argument, they will be able to obtain the polyhedrality of $K'$. 
As we will see later, in contrast to all these work in the literature, we are taking a completely different perspective and do not make use of the homogeneity property of CG closure. Key is here a characterization result for general cutting-plane closures from which a fundamental Theorem~\ref{theo: main_CG} is derived. We believe that the basic proof technique here lends itself to potentially many more classes of cutting-planes. 

Now we highlight the main results of this paper in the next section.

\subsection{Main Results}

With respect to the CG closure of general closed convex set, we have the first main result:

\begin{theorem}
\label{theo: main_CG}
Given a closed convex set $K$ in $\R^n$, then $K'$ is finitely-generated if and only if there exists a finite subset $F \subseteq \Z^n$, such that $\left\{x \in \R^n \mid f x \leq \lfloor \sigma_K(f) \rfloor, \forall f \in F \right\} \subseteq K.$
\end{theorem}


Based on this above theorem, we are able to prove the following result.
\begin{theorem}
\label{theo: motzkin_equiv}
If $K$ is a Motzkin-decomposable set, then the following statements are equivalent:
\begin{enumerate}
\item $K'$ is a finitely-generated.
\item $K'$ is a rational polyhedron.
\item $K$ has rational polyhedral recession cone.
\end{enumerate}
\end{theorem}

Here a set $K \subseteq \R^n$ is called \emph{Motzkin-decomposable} (see, e.g., \cite{goberna2010motzkin,iusem2014motzkin}), if there exist a compact convex set $C$ and a closed convex cone $D$ such that $K = C+D$. 
By Minkowski-Weyl theorem, a polyhedron is a Motzkin-decomposable set, thus the last theorem generalizes and unifies all the currently known results for rational polyhedron and compact convex set.
Moreover, using the characterization result in \cite{MR3097296} for the polyhedrality of integer hull, we can immediately obtain the next result as a corollary:
\begin{corollary}
\label{cor: main}
If $K$ is a Motzkin-decomposable set in $\R^n$ and contains integer points in its interior, then $K'$ is a rational polyhedron if and only if $\conv(K \cap \Z^n)$ is a polyhedron.
\end{corollary}

Here $\conv(K \cap \Z^n)$ is called the \emph{integer hull} of $K$.
As a footnote in \cite{dadush2014chvatal}, the authors wrote the following sentences to justify the reason why they focus on the case of a compact convex set: 
\begin{displayquote}
``
If the convex hull of integer points in a convex set is not polyhedral, then the CG closure cannot be expected to be polyhedral. Since we do not have a good understanding of when this holds for unbounded convex set, we restrict our attention here to the CG closure of compact convex sets.
"
\end{displayquote}
Therefore, our Corollary~\ref{cor: main} directly addresses their concern. 

This paper is organized as follows. In Section~\ref{sec: prelim} we present some characterizations for the polyhedrality of general cutting-plane closure and some preliminary results that will be used later. In Section~\ref{sec: CG_closedconvex}, we verify Theorem~\ref{theo: main_CG}, and in Section~\ref{sec: CG_Motzkindecom}, we verify Theorem~\ref{theo: motzkin_equiv} and Corollary~\ref{cor: main}. 

\paragraph{\textbf{\emph{Notations and assumptions.}}} For any $x \in \R^n$ and a linear subspace $L \subseteq \R^n$, we denote by $\proj_L x$ the orthogonal projection of $x$ onto $L$, and for any $X \subseteq \R^n, \proj_L X: = \{\proj_L x \mid x \in X\}$. For any set $S \subseteq \R^n, \lin(S): = S \cap (-S)$ denotes the \emph{lineality space} of $S$, which is the largest linear subspace contained in $S$.
$\cone(S): = \{\sum_{i=1}^k \lambda_i s_i \mid \forall k \in \N, \lambda_i \geq 0, s_i \in S \ \forall i \in [k]\}$ denotes the \emph{conical hull} of $S$, and $(S)_+: = \{\lambda s \mid \forall \lambda \geq 0, s \in S\}$ denotes the cone that contains all non-negative multiplication of elements in $S$.
For a linear subspace $L$, we denote by $L^\perp$ the orthogonal complement of $L$. For a closed convex set $K, \rec(K): = \{r \mid k + \lambda r \in K, \forall k \in K, \lambda \geq 0\}$ denotes the \emph{recession cone} of $K$, and $\ext(K)$ denotes the set of extreme points of $K$.
$O_\epsilon(x^*) = \{x \in \R^n \mid \|x-x^*\| \leq \epsilon\}$ denotes the $\epsilon$-ball centered at $x^*$ in its ambient space. Throughout, all norm $\|\cdot\|$ refers to the Euclidean norm. For a matrix $M$, we denote by $\ker(M)$ the kernel of $M$.

\section{Preliminary Results}
\label{sec: prelim}

The well-known Dickson's lemma will be used in our later proof, and it also played an important role in some other relevant closure papers, see, e.g., \cite{MR2969261,MR4207341,zhu2021characterization}. \begin{lemma}[Dickson's lemma \cite{dickson1913finiteness}]
\label{lem: Dickson_lemma}
For any $X \subseteq \N^n$, the partially-ordered set (poset) $(X, \leq)$ has no infinite antichain.
\end{lemma}

In order theory, an \emph{antichain (chain)} is a subset of a poset such that any two distinct elements in the subset are incomparable (comparable).

Now we define a new concept for the convergence of rays in a cone.
\begin{definition}
Given a sequence $\{\alpha^i\}_{i \in \N} \subseteq \R^n$ and $\alpha^* \neq 0 \in \R^n$, if there exists $\{\lambda_i\}_{i \in \N} > 0$ such that $\lim_{i \rightarrow \infty} \lambda_i  \alpha^i = \alpha^*$, then we say $\{\alpha^i\}$ \textbf{conically converges} to $\alpha^*$, or $\alpha^i \xrightarrow{c} \alpha^*$.
\end{definition}

For the conical convergence, we have the following easy result.
\begin{lemma}
\label{lem: conic_converge_easy}
Given a sequence $\{\alpha^i\}_{i \in \N} \subseteq \R^n$ such that $\alpha^i \xrightarrow{c} \alpha^*$ and $\alpha^i \xrightarrow{c} \beta^*$ when $i \rightarrow \infty$, then there exists $\lambda > 0,$ such that $\alpha^* = \lambda \beta^*.$
\end{lemma}

\begin{proof}
By assumption, we know there exists $\{\gamma_i\}, \{\mu_i\} \subseteq \R_+$, such that $\gamma_i \alpha^i \rightarrow \alpha^*, \mu_i \alpha^i \rightarrow \beta^*$. 
Let $\beta^i: = \mu_i \alpha^i$. Then we have: $\beta^i \rightarrow \beta^*, \frac{\gamma_i}{\mu_i} \beta^i \rightarrow \alpha^*$. Hence $\frac{\gamma_i}{\mu_i} \rightarrow \frac{\|\alpha^*\|}{\|\beta^*\|}$, and $\alpha^* = \frac{\|\alpha^*\|}{\|\beta^*\|} \beta^*$.
\end{proof}

In a recent paper \cite{zhu2021characterization}, the authors study the equivalent condition for a general cutting-plane closure to be polyhedral. In this section, we will follow the same notations and definitions as in \cite{zhu2021characterization}, and exploit the characterization results therein to derive new results for CG closure. For the completeness of this paper, we will include the proofs for those results in the Appendix. 

Given a family of cutting-planes $\alpha x \leq \beta$ for any $(\alpha, \beta) \in \Omega$, it is referred to as ``a family of cuts given by $\Omega$''. Then the corresponding \emph{(cutting-plane) closure} is defined as:
\begin{equation}
\I(\Omega) := \bigcap_{(\alpha, \beta) \in \Omega} \{x \in \R^n \mid \alpha x \leq \beta\}.
\end{equation}
Here without loss of generality (w.l.o.g.) we can assume that $(0, \ldots,0,1) \in \Omega$, since $(0, \ldots,0,1)$ corresponds to the trivial inequality $0 \cdot x \leq 1$.

For the valid inequality of the closure, we have the following result.
\begin{proposition}[Proposition~1 \cite{zhu2021characterization}]
\label{prop: valid_ineq_for_closure}
Given $\Omega \subseteq \R^{n+1}$ containing $(0, \ldots, 0, 1)$, such that $\I(\Omega) \neq \emptyset$. Then $\alpha x \leq \beta$ is a valid inequality to $\I(\Omega)$ if and only if $(\alpha, \beta) \in \cl \cone(\Omega)$.
\end{proposition}

For any set $S$, we use $\cl(S)$ to denote the smallest closed set containing $S$, which is also called \emph{closure} in topology. To avoid confusion, we will only use $\cl(S)$ to refer the topological closure.
This above proposition immediately implies the following consequence. 


\begin{corollary}
\label{cor: finitely_generated}
Given $\Omega \subseteq \R^{n+1}$ containing $(0, \ldots, 0, 1)$ with $\I(\Omega) \neq \emptyset$. Then $\I(\Omega)$ is finitely-generated if and only if there exists a finite subset $\bar \Omega \subseteq \Omega$ such that $\cone(\bar \Omega) = \cl \cone(\Omega)$. 
\end{corollary}

%

The proofs for both Proposition~\ref{prop: valid_ineq_for_closure} and Corollary~\ref{cor: finitely_generated} can be found in  Appendix~\ref{append: 1}.
From this above Corollary~\ref{cor: finitely_generated}, we know that in order to show $\I(\Omega)$ is finitely-generated, it suffices to show $\cl \cone(\Omega)$ is a polyhedral cone and can be generated by finitely many elements from $\Omega$.
The next easy lemma is helpful for characterizing $\cl \cone(\Omega)$. Here the $\oplus$ denotes the \emph{direct sum}. 

\begin{lemma}
\label{lem: easy_lemma}
For any $\Omega \subseteq \R^n$, let $L = \lin(\cl \cone(\Omega))$. Then 
$\cl \cone(\Omega) =\cl \cone(\proj_{L^\perp} \Omega) \oplus L$, where $\cl \cone(\proj_{L^\perp} \Omega)$ is a pointed, closed convex cone.
\end{lemma}
Recall that a cone is called \emph{pointed} if its lineality space is the origin.
In order to prove such result, we will also require the following lemma. 
\begin{lemma}[fact 9 \cite{studeny1993convex}]
\label{lem: cone_direct_sum}
Given a non-empty closed convex cone $K$, $K \cap \lin(K)^\perp$ is a pointed cone and $K = (K \cap \lin(K)^\perp) \oplus \lin(K)$. 
\end{lemma}

\begin{proof}[Proof of Lemma~\ref{lem: easy_lemma}]
By Lemma~\ref{lem: cone_direct_sum}, it suffices to show: $\cl \cone(\proj_{L^\perp} \Omega) = \cl \cone(\Omega) \cap L^\perp$.
First, we want to show $\cl \cone(\proj_{L^\perp} \Omega) \subseteq \cl \cone(\Omega) \cap L^\perp$. The relation $\cl \cone(\proj_{L^\perp} \Omega) \subseteq \cl \cone(L^\perp) = L^\perp$ is obvious. Moreover, for any $\omega \in \Omega, \ \proj_{L^\perp} \omega = \omega - r,$ for some $r \in L$. Hence $\proj_{L^\perp} \Omega \subseteq \Omega + L \subseteq \cl \cone(\Omega) + \cl \cone(\Omega) = \cl \cone(\Omega)$. Therefore, $\cl \cone(\proj_{L^\perp} \Omega) \subseteq \cl \cone(\Omega)$, which completes the proof of this $\subseteq$ direction. 

Then, we want to show that $\cl \cone(\proj_{L^\perp} \Omega) \supseteq \cl \cone(\Omega) \cap L^\perp$. Arbitrarily pick $x^* \in \cl \cone(\Omega) \cap L^\perp$. If $x^* \in \cone(\Omega)$, then $x^* = \proj_{L^\perp} x^* \in \proj_{L^\perp} \cone(\Omega) = \cone(\proj_{L^\perp} \Omega)$. If $x^i \rightarrow x^*$ for a sequence of $\{x^i\} \subseteq \cone(\Omega)$, then $\proj_{L^\perp} x^i \rightarrow x^*$ where $\proj_{L^\perp} x^i \in \proj_{L^\perp} \cone(\Omega) = \cone(\proj_{L^\perp} \Omega).$ Hence $x^* \in \cl \cone(\proj_{L^\perp} \Omega)$. This completes the proof.
\end{proof}

It is well-known that, a pointed, closed convex cone is a polyhedral cone, if and only if it has finitely many different extreme rays. For a pointed $\cl \cone(\Omega)$, its extreme rays can be exactly characterized by elements in $\Omega$, as stated by the next lemma. We include its proof in Appendix~\ref{append: extremeray_char}.
\begin{lemma}[Corollary~1, Lemma~3 \cite{zhu2021characterization}]
\label{lem: characterization_extremeray}
Given $\Omega \subseteq \R^{n+1}$ with $(0, \ldots, 0, 1) \in \Omega$ and $0 \notin \Omega$. If $\cl \cone(\Omega)$ is pointed, then for any extreme ray $r$ of $\cl \cone(\Omega)$, either $r \in (\Omega)_+$, or there exist different $\{r^i\} \subseteq \Omega$ such that $r^i \xrightarrow{c} r$. 
\end{lemma}

Henceforth, when we mention a ray $r$ of a cone, we will make no distinction between $r$ and its positive scalar multiplication. In other words, we say two rays $r^1$ and $r^2$ are different, if and only if there does not exist $\lambda > 0$, such that $r^1 = \lambda r^2$.





\section{Chv\'atal-Gomory Closure of Closed Convex Set.}
\label{sec: CG_closedconvex}

We will prove Theorem~\ref{theo: main_CG} in this section.

For a given closed convex set $K$, we denote the family of CG cuts of $K$ to be:
\begin{equation}
\Omega_{\CG}: = \{(0, \ldots, 0, 1)\} \cup \{(c, \lfloor \sigma_K(c) \rfloor), \forall c \in \Z^n\}. 
\end{equation}
Then by definition of CG closure, there is $K' = \I(\Omega_\CG)$. For ease of notation, when it is clear from the context, we will not specify what is the corresponding closed convex set $K$ of $\Omega_\CG$. 
Throughout, a CG cut $cx \leq \lfloor \sigma_K(c) \rfloor$ will sometimes also be referred to as a vector $(c, \lfloor \sigma_K(c) \rfloor)$. 

Before presenting the proof for the main Theorem~\ref{theo: main_CG}, we will need the following lemmas.

\begin{lemma}[Gordan's lemma]
\label{lem: hilbert_basis}
Given a lattice $L \subseteq \Z^n$ and a rational polyhedral cone $C \subseteq \R^n$. Then there exists a finite set of lattice points $\{g^1, \ldots, g^m\} \subset C \cap L$ such that every point $x \in C \cap L$ is an integer conical combination of these points: $x = \sum_{j=1}^m \lambda_j g^j, \lambda_j \in \N$ for all $j \in [m].$
\end{lemma}
Here the finite generator $\{g^1, \ldots, g^m\}$ of $C \cap L$ in the above lemma is usually referred to as the \emph{Hilbert basis} of $C$ (see, e.g., \cite{cook1986integer}). From Gordan's lemma we obtain the next result.

\begin{lemma}
\label{lem: infinite_chain}
Given a rational polyhedral cone $C \subseteq \R^n$, a sequence of integer vectors $\{v^i\}_{i \in \N} \subseteq C \cap \Z^n$, and a rational vector $q^* \in \Q^n$. Then there must exist an infinite set $I \subseteq \N$ and $i^* \in \N$, such that for any $i \in I, v^i - v^{i^*} \in C$, and $v^i q^* - \lfloor v^i q^* \rfloor = v^{i^*} q^* - \lfloor v^{i^*} q^* \rfloor$.
\end{lemma}
\begin{proof}
By Lemma~\ref{lem: hilbert_basis}, we know there exist $g^1, \ldots, g^m \in C \cap \Z^n$, such that $x \in C \cap \Z^n$ if and only if $x$ can be written as the integer conical combination of these points. Therefore, for each $v^i$, there exists $\lambda^i \in \N^m$ such that $v^i = \sum_{j=1}^m \lambda^i_j g^j$. 

\smallskip \noindent
\textsc{\textbf{Folklore}}
\emph{An infinite poset contains either an infinite chain or an infinite antichain.}
\smallskip

Within the infinite poset $\Lambda: = \{\lambda^i\}_{i \in \N} \subseteq \N^m$ ordered by component-wise order $\leq$, from the Dickson's Lemma~\ref{lem: Dickson_lemma} and this above folklore, 
we know there must exist an infinite index set $I' \subseteq \N$, such that $\{\lambda^i\}_{i \in I'}$ is an infinite chain within $\N^m.$

Since $q^*$ is a rational vector, we can write it as $q^*: = \frac{1}{D}z$, where $z \in \Z^n$, and $D$ is the least common multiple of the denominators of each $q_1, \ldots, q_n$. So for any $v^i \in \Z^n, v^i q^*$ can be written as $\frac{1}{D} \cdot v^i z$ , where $c^i z \in \Z$. Therefore, $\{v^i q^* - \lfloor v^i q^* \rfloor \mid i \in I'\} \subseteq  \{0, \frac{1}{D}, \ldots, \frac{D-1}{D}\}$, which is a finite set. Here $I'$ is an infinite index set, by the pigeonhole principle, there also exists another infinite index set $I \subseteq I'$, such that $v^i q^* - \lfloor v^i q^*\rfloor = v^j q^* - \lfloor v^j q^*\rfloor$ for any $i,j \in I.$ 

So far we have obtained an infinite index set $I$, such that for any $i \in I, v^i = \sum_{j=1}^m \lambda^i_j g^j$, $\{\lambda^i\}_{i \in I}$ is an infinite chain within $\N^m$, and $v^i q^* - \lfloor v^i q^*\rfloor = v^j q^* - \lfloor v^j q^*\rfloor$ for any $i,j \in I$. Since $\{\lambda^i\}_{i \in I}$ is an infinite chain within $\N^m$, then there must exist $i^* \in I$ such that $\lambda^{i^*}$ is the \textbf{least element} (a.k.a. minimum element) of $\{\lambda^i\}_{i \in I}$. 
Therefore, for any $i \in I$, $v^i q^* - \lfloor v^i q^* \rfloor = v^{i^*} q^* - \lfloor v^{i^*} q^* \rfloor$ and $\lambda^i \geq \lambda^{i^*}$, which implies that $v^i - v^{i^*} = \sum_{j=1}^m (\lambda^i_j - \lambda^{i^*}_j) g^j \in C$.
\end{proof}

The next lemma states that, within any infinite sequence of CG cuts of $K$, there must exist a conically convergent subsequence which converges to a valid inequality of $K$.

\begin{lemma}
\label{lem: conical_convergence_CG}
Given a closed convex set $K$ and a sequence $\{(r^i, \lfloor \sigma_K(r^i) \rfloor)\}_{i \in \N} \subseteq \Omega_\CG$.
Then there exists a subsequence $\{(r^i, \lfloor \sigma_K(r^i) \rfloor)\}_{i \in I}$, such that when $i \in I, i \rightarrow \infty$, $(r^i, \lfloor \sigma_K(r^i) \rfloor) \xrightarrow{c} (r^*, r^*_0)$ for some valid inequality $r^* x \leq r^*_0$ of $K$.
\end{lemma}

\begin{proof}
Picking $\gamma_i: = \frac{1}{\|(r^i, \lfloor \sigma_K(r^i) \rfloor)\|}$ for all $i \in \N$. Then $\gamma_i (r^i, \lfloor \sigma_K(r^i) \rfloor) \in O_1(0)$ which is a compact set. By Bolzano-Weierstrass theorem, we can find a convergent subsequence $\{\gamma_i (r^i, \lfloor \sigma_K(r^i) \rfloor)\}_{i \in I}$. Denote $\gamma_i (r^i, \lfloor \sigma_K(r^i) \rfloor) \rightarrow (r^*, r^*_0)$. 
Since each $r^i \in \Z^n$, we know there exists an infinite subsequence of $\{r^i\}_{i \in I}$ such that $\|r^i\| \rightarrow \infty$. W.l.o.g. we assume $\|r^i\| \rightarrow \infty$ when $i \rightarrow \infty$, since $\gamma_i r^i \rightarrow r^*$, then there is $\gamma_i \rightarrow 0$. Furthermore, because $\gamma_i \lfloor \sigma_K(r^i) \rfloor \rightarrow r^*_0$ and $\gamma_i ( \sigma_K(r^i)  - 1) < \gamma_i \lfloor \sigma_K(r^i) \rfloor \leq \gamma_i  \sigma_K(r^i) $, we know that $\gamma_i  \sigma_K(r^i)  \rightarrow r^*_0$. Now we want to argue that $r^* x \leq r^*_0$ is valid to $K$. If not, then there exists $x^* \in K$ such that $r^* x^* > r^*_0$. Denote $\epsilon: = r^* x^* - r^*_0 > 0$. Since $\gamma_i (r^i, \sigma_K(r^i)) \rightarrow (r^*, r^*_0)$, then there exist $N_0 \in \N$ and $\kappa: = \frac{\epsilon}{2(1+\|x^*\|)}$, such that when $i \geq N_0:$ 
$$
\|r^* - \gamma_i r^i\| \leq \kappa, \quad |r^*_0 - \gamma_i \sigma_K(r^i)| \leq \kappa.
$$
Therefore, when $i \geq N_0$:
\begin{align*}
\begin{split}
 \gamma_i r^i x^* - \gamma_i \sigma_K(r^i) & \geq r^* x^* - \kappa \|x^*\| - r^*_0 - \kappa \\
& = \epsilon - \kappa(1 + \|x^*\|) \\
& = \frac{1}{2} \epsilon > 0.
\end{split}
\end{align*}

This gives the contradiction since $r^i x \leq \sigma_K(r^i)$ is valid to $K$. 
\end{proof}


Next we present the most crucial result for establishing the proof of the main theorem.
\begin{proposition}
\label{prop: awesome_prop}
Given a closed convex set $K \subseteq \R^n$ and a rational polyhedron $P \subseteq K$ such that $P = \left\{ x\in \R^n \mid fx \leq \lfloor \sigma_K(f) \rfloor, \forall f \in F\right\}$ for some finite set $F \subseteq \Z^n$. If $\{(c^i, \lfloor \sigma_K(c^i) \rfloor)\}_{i \in \N} \subseteq \Omega_\CG$ is a sequence of vectors with $\sigma_P(c^i) > \lfloor \sigma_K(c^i) \rfloor$ for any $i \in \N$, then there exist a finite set $\Lambda \subseteq  \Omega_\CG$ and an infinite index set $I \subseteq \N$, such that $(c^i, \lfloor \sigma_K(c^i) \rfloor) \in \cone(\Lambda)$ for any $i \in I$.  
\end{proposition}

\begin{proof}
For each $i \in \N$, since $P \subseteq K$, there is $\sigma_P(c^i) \leq \sigma_K(c^i) < \infty$, we know there exists extreme point $p^i$ of $P$, such that $c^i p^i = \sigma_P(c^i)$. So from the condition of this proposition, we have the following inequalities:
\begin{equation*}
\sigma_K(c^i) \geq c^i p^i = \sigma_P(c^i) > \lfloor \sigma_K(c^i) \rfloor.
\end{equation*}
Hence $c^i p^i > \lfloor \sigma_K(c^i) \rfloor = \lfloor c^i p^i \rfloor$, for all $i \in \N$. 
This can be visualized in Fig.~\ref{fig: 1}.
Since the number of extreme points of polyhedron $P$ is finite, by the pigeonhole principle, we know there exist a single extreme point $p^*$ of $P$ and an infinite subset $I_1 \subseteq \N$, such that $p^i = p^*$ for any $i \in I_1$. 
Note that for a rational polyhedron $P$ and an extreme point $p^* \in P$, $c^i p^* = \sigma_{P}(c^i)$ if and only if 
$$
c^i \in C: = \{x \in \R^n \mid (p^* - p) x \geq 0 \ \forall p \in \ext(P), r x \leq 0 \ \forall r \in \rec(P)\},
$$
where $C$ is a rational polyhedral cone. For the rational vector $p^*$ and rational polyhedral cone $C$, by Lemma~\ref{lem: infinite_chain}, we know there exist another infinite subset $I \subseteq I_1$ and $i^* \in I_1$, such that for any $i \in I, c^i - c^{i^*} \in C$, and $c^i p^* - \lfloor c^i p^* \rfloor = c^{i^*} p^* - \lfloor c^{i^*} p^* \rfloor$. Now we denote
$$
\Lambda: = \{ (f, \lfloor \sigma_K(f) \rfloor) \ \forall f \in F, (c^{i^*}, \lfloor \sigma_K (c^{i^*}) \rfloor), (0, \ldots, 0, 1) \}.
$$
Here we have $\Lambda \subseteq \Omega_\CG$ and $\Lambda$ is a finite set.

Lastly, we want to show that the above constructed $I$ and $\Lambda$ satisfy the condition of this proposition, namely, for any $i \in I$, there is $(c^i, \lfloor \sigma_K(c^i) \rfloor) \in \cone(\Lambda)$. By condition of $I$ and $i^*$, we know $c^{i} - c^{i^*} \in C$, which means $(c^{i} - c^{i^*}) x \leq (c^{i} - c^{i^*}) p^*$ is valid to $P$. By definition of $P$, so this implies that $(c^{i} - c^{i^*}, c^{i} p^*- c^{i^*} p^*) \in \cone(\{ (f, \lfloor \sigma_K(f) \rfloor) \ \forall f \in F, (0, \ldots, 0, 1) \})$. 
Note that for any $i \in I, c^i p^* - \lfloor c^i p^* \rfloor = c^{i^*} p^* - \lfloor c^{i^*} p^* \rfloor$, therefore:
\begin{align*}
\begin{split}
(c^i, \lfloor \sigma_K(c^i) \rfloor) & = (c^i, \lfloor c^i p^* \rfloor)\\
& = (c^{i^*}, \lfloor c^{i^*} p^* \rfloor) + (c^{i} - c^{i^*}, c^{i} p^*- c^{i^*} p^*)  \\
 & = (c^{i^*}, \lfloor \sigma_K (c^{i^*}) \rfloor) + (c^{i} - c^{i^*}, c^{i} p^*- c^{i^*} p^*) \\
& \in \cone(\Lambda).
\end{split}
\end{align*}
Hence we complete the proof.
\end{proof}

\begin{figure}
\begin{center}
\begin{tikzpicture}[scale=1.6]
 \draw[thick] (0,1.4) arc(90:270:2cm and 1cm);
  \fill[green, opacity = 0.2] (0,1.4) arc(90:270:2cm and 1cm);
  \fill[green, opacity = 0.2] (0,1.4) -- (0.5, 1.4) -- (0.5, -0.6) -- (0,-0.6);
     \draw[thick] (0.3, 1.2) -- (-1.2, 1) -- (-1.7, 0.4) coordinate (a_1) -- (-0.8, -0.3) -- (0.1, -0.4);
     \node at (-1.7, 0.4)[circle,fill,inner sep=1.3pt]{};
     \draw (-1.65, 0.4) node [left]{$p^*$};
     \draw (-0.3,0.5) node  [right] {$P$};
          \draw (0.2,0.2) node  [right] {$K$};
\fill[orange, opacity = 0.2] (0.3, 1.2) -- (-1.2, 1) -- (-1.7, 0.4) coordinate (a_1) -- (-0.8, -0.3) -- (0.1, -0.4);
\draw[thick, dashed, red] (-2.8, 1.2) -- (-1.09, -1);
\draw[thick, red] (-2.1, 1.2) -- (-0.39, -1);
\draw[thick, dashed, blue] (-2.32, 1.2) -- (-0.61, -1);
\draw[thick, dashed, red] (-1.9, 1.4) -- (-2.15, -0.9);
\draw[thick, dashed, blue] (-1.6, 1.4) -- (-1.85, -0.9);
\draw[thick, red] (-1.3, 1.4) -- (-1.55, -0.9);
\draw[->, thick] (-1.12, -0.95) -- (-0.85, -0.74);
\draw[->, thick] (-2.12, -0.76) -- (-1.7, -0.8);
    \end{tikzpicture}
    \end{center}
    \caption{The red solid lines represent the CG cuts of $K$ that cut off extreme point $p^*$, and the dashed red and blue lines represent the corresponding valid inequalities of $K$ and $P$. Then these two CG cuts of $K$ are also the CG cuts of $P$.}
    \label{fig: 1}
\end{figure}
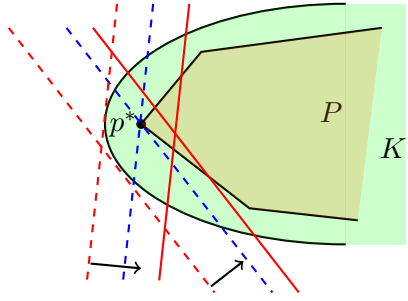

The following is the last piece of result we will need to prove Theorem~\ref{theo: main_CG}.

\begin{proposition}
\label{prop: affine_rational_whysohard}
Given a closed convex set $K$ in $\R^n$, and there exists a finite subset $F \subseteq \Z^n$, such that $\left\{x \in \R^n \mid f x \leq \lfloor \sigma_K(f) \rfloor, \forall f \in F \right\} \subseteq K.$
Then for any $v \in \cl \cone(\Omega_\CG)$, there exists a finite set $\Lambda_v \subseteq \Omega_\CG$, such that $v \in \cone(\Lambda_v)$.
\end{proposition}

\begin{proof}
If $v \in \cone(\Omega_\CG)$, then by Carath\'eodory's theorem, there exists a finite subset $\Lambda \subseteq \Omega_\CG$ with at most $\dim(\Omega_\CG)$ elements, such that $v \in \cone(\Lambda)$. 

If $v \notin \cone(\Omega_\CG)$, since $v \in L \subseteq \cl \cone(\Omega_\CG)$, we can find a sequence in $\cone(\Omega_\CG)$ converging to $v$. 
Assume $\sum_{j=1}^d \lambda_{i,j} v^{i,j} \rightarrow v$ when $i \rightarrow \infty$, here $d = \dim(\Omega_\CG)$ and $v^{i,j} \in \Omega_\CG, \lambda_{i,j} \geq 0$ for all $i \in \N, j \in [d]$.
Denote $P = \left\{x \in \R^n \mid f x \leq \lfloor \sigma_K(f) \rfloor, \forall f \in F \right\}$, which is contained in $K$.
First, within the set $\{v^{i,1}\}_{i \in \N} \subseteq \Omega_\CG$, there must exist an infinite index subset $I_1 \subseteq \N$ such that CG cuts within $\{v^{i,1}\}_{i \in I_1}$ are either \textbf{all} valid to $P$, or \textbf{all} invalid to $P$. If $\{v^{i,1}\}_{i \in I_1}$ all correspond to valid inequalities of $P$, then $\{v^{i,1}\}_{i \in I_1} \subseteq \cone(\{(f, \lfloor \sigma_K(f) \rfloor) \ \forall f \in F, (0, \ldots, 0,1)\})$, where $\{(f, \lfloor \sigma_K(f) \rfloor) \ \forall f \in F, (0, \ldots, 0,1)\} \subseteq \Omega_\CG$. If they all correspond to invalid inequalities of $P$, then by Proposition~\ref{prop: awesome_prop}, there exists another infinite index set $\bar{I}_1 \subseteq I_1$ and finite set $\Lambda_1 \subseteq \Omega_\CG$, such that $v^{i,1} \in \cone(\Lambda_1)$ for all $i \in \bar{I}_1$. In other words, no matter whether CG cuts within $\{v^{i,1}\}_{i \in I_1}$ are all valid to $P$ or not, we can always find an infinite index set $\bar{I}_1 \subseteq I_1$ and finite set $\Lambda_1 \subseteq \Omega_\CG$, such that $v^{i,1} \in \cone(\Lambda_1)$ for all $i \in \bar{I}_1$.
Now, within the set $\{v^{i,2}\}_{i \in \bar{I}_1} \subseteq \Omega_\CG$, we can do the above argument one more time, and obtain another infinite index subset $\bar{I}_2 \subseteq \bar{I}_1$ and another finite set $\Lambda_2 \subseteq \Omega_\CG$, such that $v^{i,2} \in \cone(\Lambda_2)$ for all $i \in \bar{I}_2$. In fact, since $\bar{I}_2 \subseteq \bar{I}_1$, we also have  $v^{i,1} \in \cone(\Lambda_1)$ for all $i \in \bar{I}_2$. After doing such argument for $d$ times, eventually, we will obtain an infinite index set $\bar{I}_d \subseteq \N$ and $d$ finite sets $\Lambda_1, \ldots, \Lambda_d \subseteq \Omega_\CG$, such that $v^{i,j} \in \cone(\Lambda_j)$ for any $i \in \bar{I}_d, j \in [d]$. Note that $\sum_{j=1}^d \lambda_{i,j} v^{i,j} \rightarrow v$ when $i \in \bar{I}_d, i \rightarrow \infty$, and $\sum_{j=1}^d \lambda_{i,j} v^{i,j} \in \cone(\Lambda_1 \cup \ldots \cup \Lambda_d)$, for any $i \in \bar{I}_d$. Therefore, we obtain $v \in \cone(\Lambda_1 \cup \ldots \cup \Lambda_d)$. Here $\Lambda_1 \cup \ldots \cup \Lambda_d \subseteq \Omega_\CG$ is a finite set, by picking $\Lambda_v : = \Lambda_1 \cup \ldots \cup \Lambda_d$ we conclude the proof.
\end{proof}

Now we are ready to verify the main theorem in this section. 

%
%

\begin{proof}[Proof of Theorem~\ref{theo: main_CG}]
It suffices for us to show the ``if'' direction: if there exists a finite subset $F \subseteq \Z^n$ such that $\left\{x \in \R^n \mid f x \leq \lfloor \sigma_K(f) \rfloor, \forall f \in F \right\} \subseteq K$, then $K'$ is finitely-generated. 
Denote $P = \left\{x \in \R^n \mid f x \leq \lfloor \sigma_K(f) \rfloor, \forall f \in F \right\}$, and $L = \lin(\cl \cone(\Omega_\CG))$.

By Lemma~\ref{lem: easy_lemma}, $\cl \cone(\Omega_\CG) = \cl \cone(\proj_{L^\perp} \Omega_\CG) \oplus L$ and $\cl \cone(\proj_{L^\perp} \Omega_\CG)$ is a pointed, closed convex cone.
We start our argument by analyzing the extreme rays of $\cl \cone(\proj_{L^\perp} \Omega_\CG)$.
By Lemma~\ref{lem: characterization_extremeray}, any extreme ray of $\cl \cone(\proj_{L^\perp} \Omega_\CG)$ is either in $(\proj_{L^\perp} \Omega_\CG)_+$, or can be conically converged by elements from $\proj_{L^\perp} \Omega_\CG$. Define $\Omega^\uparrow$ as the set of extreme rays of $\cl \cone(\proj_{L^\perp} \Omega_\CG)$ that can be conically converged by elements from $\proj_{L^\perp} \Omega_\CG$, and define $\Omega^*$ as the set of extreme rays of $\cl \cone(\proj_{L^\perp} \Omega_\CG)$ that are in $(\proj_{L^\perp} \Omega_\CG)_+$.

First we consider the set $\Omega^\uparrow$.
Let $(r^*, r^*_0) \in \Omega^\uparrow$. By assumption of vectors in $\Omega^\uparrow$, we know there exists a sequence $\{(r^i, \lfloor \sigma_K(r^i) \rfloor)\}_{i \in \N} \subseteq \Omega_\CG$ and $\{\gamma_i\}_{i \in \N} \subseteq \R_+$, such that 
\begin{equation}
\label{eq: cc1}
\gamma_i \proj_{L^\perp} (r^i, \lfloor \sigma_K(r^i) \rfloor) \rightarrow (r^*, r^*_0) \text{ when } i \rightarrow \infty.
\end{equation}
By Lemma~\ref{lem: conical_convergence_CG}, within this sequence $\{(r^i, \lfloor \sigma_K(r^i) \rfloor)\}_{i \in \N}$, there exists a subsequence $\{(r^i, \lfloor \sigma_K(r^i) \rfloor)\}_{i \in I}$, such that when $i \in I, i \rightarrow \infty,$ $(r^i, \lfloor \sigma_K(r^i) \rfloor) \xrightarrow{c} (v^*, v^*_0)$ for some valid inequality $v^* x \leq v^*_0$ of $K$. Let 
\begin{equation}
\label{eq: cc2}
\lambda_i (r^i, \lfloor \sigma_K(r^i) \rfloor) \rightarrow (v^*, v^*_0) \text{ when } i \in I, i \rightarrow \infty.
\end{equation}
Here each $\lambda_i > 0$. 
We can rewrite \eqref{eq: cc2} as follows:
\begin{equation}
\label{eq: cc3}
\lambda_i \proj_{L^\perp} (r^i, \lfloor \sigma_K(r^i) \rfloor) + \lambda_i \proj_L (r^i, \lfloor \sigma_K(r^i) \rfloor)  \rightarrow (v^*, v^*_0) \text{ when } i \in I, i \rightarrow \infty.
\end{equation}
If $(v^*, v^*_0) \in L$, then for each $i \in I$, take the inner product of both sides of \eqref{eq: cc3} with the vector $\frac{\gamma_i^2}{\lambda_i} \proj_{L^\perp} (r^i, \lfloor \sigma_K(r^i) \rfloor)$, this gives us $\gamma_i^2 \|\proj_{L^\perp} (r^i, \lfloor \sigma_K(r^i) \rfloor)\|^2 \rightarrow 0$. Together with \eqref{eq: cc1} we get the contradiction, since $(r^*, r^*_0) \neq 0$. Hence $(v^*, v^*_0) \notin L$, and $\proj_{L^\perp} (v^*, v^*_0) \neq 0$. From $\lambda_i (r^i, \lfloor \sigma_K(r^i) \rfloor) \rightarrow (v^*, v^*_0)$ for $i \in I, i \rightarrow \infty$, we simply obtain that $\proj_{L^\perp} (r^i, \lfloor \sigma_K(r^i) \rfloor) \xrightarrow{c} \proj_{L^\perp} (v^*, v^*_0)$, for $i \in I, i \rightarrow \infty$. Since there is also $\proj_{L^\perp} (r^i, \lfloor \sigma_K(r^i) \rfloor) \xrightarrow{c} (r^*, r^*_0)$, by Lemma~\ref{lem: conic_converge_easy}, we know $(r^*, r^*_0) = \lambda \proj_{L^\perp} (v^*, v^*_0)$, for some $\lambda > 0$. Since $v^* x \leq v^*_0$ is valid to $K$, which contains $P$,
so
$(v^*, v^*_0) \in \cone(\bar{\Omega}^\uparrow),$
where
$$\bar{\Omega}^\uparrow: = \{ (f, \lfloor \sigma_K(f) \rfloor) \ \forall f \in F, (0, \ldots, 0, 1)\} \subseteq \Omega_\CG.$$
From $(r^*, r^*_0) = \lambda \proj_{L^\perp} (v^*, v^*_0)$, we also have $(r^*, r^*_0) \in \cone (\proj_{L^\perp} \bar{\Omega}^\uparrow)$. 
Since here $(r^*, r^*_0) \in \Omega^\uparrow$ is arbitrary,
in the end, we have shown $\Omega^\uparrow \subseteq \cone (\proj_{L^\perp} \bar{\Omega}^\uparrow)$, for some finite subset $\bar{\Omega}^\uparrow$ of $\Omega_\CG$.

Now we consider the other set $\Omega^*$. Assuming $\Omega^*$ contains infinitely many \textbf{different} extreme rays of $\cl \cone(\proj_{L^\perp} \Omega_\CG)$: let $\{(r^i, r^i_0)\}_{i \in \N} \subseteq \Omega^*$ be one sequence of different extreme rays of $\cl \cone(\proj_{L^\perp} \Omega_\CG)$, where $(r^i, r^i_0) = \proj_{L^\perp} (c^i, \lfloor \sigma_K(c^i) \rfloor)$ for CG cut $c^i x \leq \lfloor \sigma_K(c^i) \rfloor, c^i \in \Z^n$. Since $(r^i, r^i_0)$ is an extreme ray, then $(r^i, r^i_0) \notin \cone (\proj_{L^\perp} \bar{\Omega}^\uparrow)$, so there is also $(c^i, \lfloor \sigma_K(c^i) \rfloor) \notin \cone( \bar{\Omega}^\uparrow)$. By Proposition~\ref{prop: valid_ineq_for_closure} and definition of $\bar{\Omega}^\uparrow$, this implies that inequality $c^i x \leq \lfloor \sigma_K(c^i) \rfloor$ is not valid to $P$, for any $i \in \N$.
By Proposition~\ref{prop: awesome_prop}, we know there exists a finite set $\Lambda \subseteq \Omega_\CG$ and an infinite set $I' \subseteq \N$, such that $(c^i, \lfloor \sigma_K(c^i) \rfloor) \in \cone(\Lambda)$ for any $i \in I'$. 
Hence
$
(r^i, r^i_0)= \proj_{L^\perp} (c^i, \lfloor \sigma_K(c^i) \rfloor) 
\in \proj_{L^\perp} \cone(\Lambda)
= \cone(\proj_{L^\perp} \Lambda) \subseteq \cone(\proj_{L^\perp} \Omega_\CG).
$
However, by our above assumption, for any $i \in I', (r^i, r^i_0) \in \Omega^*$ is extreme ray of $\cl \cone(\proj_{L^\perp} \Omega_\CG)$, we get the contradiction. So $\Omega^*$ can only contain finitely many different extreme rays of $\cl \cone(\proj_{L^\perp} \Omega_\CG)$.
By assumption of $\Omega^*$, here we can find a finite subset $\bar{\Omega}^* \subseteq \Omega_\CG$, such that $\Omega^* \subseteq \cone(\proj_{L^\perp} \bar{\Omega}^*)$.

So far, we have shown that, there exists finite subsets $\bar{\Omega}^\uparrow$ and $\bar{\Omega}^*$ of $\Omega_\CG$, such that $\Omega^\uparrow \cup \Omega^* \subseteq \cone \big(\proj_{L^\perp} (\bar{\Omega}^\uparrow \cup \bar{\Omega}^*)\big)$. Since $\Omega^\uparrow \cup \Omega^*$ contains all the extreme rays of $\cl \cone(\proj_{L^\perp} \Omega_\CG)$, we have:
$$
\cl \cone(\proj_{L^\perp} \Omega_\CG) = \cone \big(\proj_{L^\perp} (\bar{\Omega}^\uparrow \cup \bar{\Omega}^*)\big).
$$

For the lineality space $L$, which is a subset of $\cl \cone(\Omega)$, by Proposition~\ref{prop: affine_rational_whysohard}, we can find a finite subset $\bar{\Omega}_L \subseteq \Omega_\CG$, such that $L \subseteq \cone(\bar{\Omega}_L)$. 
Hence:
$$
\cl \cone(\Omega_\CG) = \cl \cone(\proj_{L^\perp} \Omega_\CG) \oplus L \subseteq \cone \big(\bar{\Omega}^\uparrow \cup \bar{\Omega}^* \cup \bar{\Omega}_L \big).
$$
Therefore, we obtain that: $\cl \cone(\Omega_\CG) = \cone \big(\bar{\Omega}^\uparrow \cup \bar{\Omega}^* \cup \bar{\Omega}_L \big)$, where $\bar{\Omega}^\uparrow, \bar{\Omega}^*$ and $\bar{\Omega}_L$ are finite subsets of $\Omega_\CG$.
By Corollary~\ref{cor: finitely_generated} we conclude the proof.
\end{proof}

\section{Chv\'atal-Gomory Closure of Motzkin-Decomposable Set.}
\label{sec: CG_Motzkindecom}

In this section, we will prove that, the CG closure of a Motzkin-decomposable set is a rational polyhedron if and only if it has rational polyhedral recession cone. 
Before presenting the proof for such main result, we first develop some intuition by examining the following examples.
As we shall see in a moment, in some sense, the assumptions of Motzkin-decomposable is necessary.

\begin{example}
\label{exam: hyper1}
Consider closed convex set $$K_1 = \{x \in \R^2_+ \mid x_1 \cdot x_2 \geq 2\},$$ see fig.~\ref{fig:sub1}. 
Note that $\rec(K_1) = \R^2_+$ a rational polyhedral cone, but $K_1$ is not Motzkin-decomposable since there does not exist a compact convex set $C$ such that $K_1 = C+\R^2_+$. 

Moreover, $K_1$ has integer hull $$\conv(K_1 \cap \Z^2) = \{x \in \R^2 \mid x_1 + x_2 \geq 3, x_1 \geq 1, x_2 \geq 1\},$$ while $K_1'$ is not finitely-generated. To observe this, realize that $\conv(K_1 \cap \Z^2) \subseteq K_1' \subseteq K_1$ where $\rec(K_1) = \rec(\conv(K_1 \cap \Z^2)) = \R^2_+$, so $\rec(K_1') = \R^2_+$. 
If $K_1'$ is finitely-generated, then $K_1'$ will have facet-defining inequalities $x_1 \geq \alpha_1, x_2 \geq \alpha_2$ for some $\alpha_1, \alpha_2 \geq 0$, and inequalities $x_1 \geq \alpha_1, x_2 \geq \alpha_2$ are both CG cuts of $K_1$. Clearly $\alpha_1 = \alpha_2 = 1$. However, there does not exist any fractional $\beta_1, \beta_2 \in (0,1)$, such that $x_1 \geq \beta_1$ and $x_2 \geq \beta_2$ are valid to $K_1$, so $x_1 \geq 1, x_2 \geq 1$ cannot be CG cuts of $K_1$, which means $K_1'$ is not finitely-generated.
\end{example}

\begin{example}
\label{exam: hyper2}
Consider another closed convex set $$K_2 = \{x \in \R^2 \mid (x_1-0.2) \cdot (x_2 - 0.2) \geq 2, x_1 > 0.2, x_2 > 0.2\},$$ see fig.~\ref{fig:sub2}. Note that $K_2$ here can be obtained by translating the closed convex set $K_1$ in Example~\ref{exam: hyper1}. Here $K_2$ is also not Motzkin-decomposable, and 
$$\conv(K_2 \cap \Z^2) = \{x \in \R^2 \mid x_1 + x_2 \geq 4, x_1 \geq 1, x_2 \geq 1\}.$$
However, $x_1 \geq 0.2, x_2 \geq 0.2, x_1 + x_2 \geq 0.4+2\sqrt{2}$ are all valid inequalities of $K_2$, so $x_1 \geq 1, x_2 \geq 1, x_1 + x_2 \geq 4$ are CG cuts of $K_2$. Therefore, $K_2'$ is finitely-generated. 
\end{example}

\begin{figure}
\centering
\label{fig: counter-exam1}
\begin{subfigure}{.4\textwidth}
\centering
\vspace*{-1.2cm}
\begin{tikzpicture}[scale=0.5]
   \draw[->] (0,0) -- (5.8,0) node[right] {$x_1$};
   \draw[->] (0,0) -- (0,5.8) node[above] {$x_2$};
   \draw[scale=1.0,domain=0.4:5.5,smooth, variable=\x,blue, thick] plot ({\x},{2/(\x)});
      \fill[opacity = 0.2, green,domain=0.36:5.8,smooth, variable=\x,blue, thick] plot ({\x},{2/(\x)});
      \fill[opacity = 0.2, green,smooth, variable=\x,blue, thick] (0.36, 5.5556) -- (5.8, 5.5556) -- (5.8, 0.3483);
  \draw[scale=1.0,domain=0.4:5.5,smooth, variable=\x,red] (1,5.8) -- (1,2) -- (2,1) -- (6.1, 1);
  \fill[opacity = 0.2, red, variable=\x,red, thick] (1,5.8) -- (1,2) -- (2,1) -- (6.1, 1) -- (6.1, 5.8);
      \foreach \x in {0,1,2,3,4,5}
       \foreach \y in {0,1,2,3,4,5}
\node at (\x, \y)[circle,fill,inner sep=0.6pt]{};
\end{tikzpicture}
\caption{The blue region is $K_1$, whose integer hull is marked in red and it is rational polyhedral, while $K_1'$ is not.}
  \label{fig:sub1}
\end{subfigure}%
\hspace*{2cm}
\begin{subfigure}{.4\textwidth}
\begin{tikzpicture}[scale=0.5]
   \draw[->] (0,0) -- (5.8,0) node[right] {$x_1$};
   \draw[->] (0,0) -- (0,5.8) node[above] {$x_2$};
   \draw[dashed] (0.2,0) node[below] {0.2} -- (0.2,5.7);
      \draw[dashed] (0,0.2) node[left] {0.2} -- (5.7,0.2);
   \draw[domain=0.6:5.5,smooth, variable=\x,blue, thick] plot ({\x},{2/(\x-0.2)+0.2});
      \fill[opacity = 0.2, green,domain=0.58:5.7,smooth, variable=\x,blue, thick] plot ({\x},{2/(\x-0.2)+0.2});
      \fill[opacity = 0.2, green,smooth, variable=\x,blue, thick] (0.58, 5.4632) -- (5.7, 5.4632) -- (5.7, 0.5636);
        \draw[domain=0.4:5.5,smooth, variable=\x,red] (1,6.0) -- (1,3) -- (3,1) -- (6.1, 1);
  \fill[opacity = 0.2, red, variable=\x,red, thick] (1,6.0) -- (1,3) -- (3,1) -- (6.1, 1) -- (6.1, 6.0);
   \draw[blue, dashed, thick]  (-0.2, 3.4284) -- (1.6142, 1.6142) -- (3.4284, -0.2);
   \draw[red, thick] (-0.2, 4.2) -- (4.2, -0.2);
      \foreach \x in {0,1,2,3,4,5}
       \foreach \y in {0,1,2,3,4,5}
\node at (\x, \y)[circle,fill,inner sep=0.6pt]{};
\end{tikzpicture}
\caption{The blue region denotes $K_2$, red region denotes $\conv(K_2 \cap \Z^2)$, and dashed lines $x_1 = \frac{1}{2}$ and $x_2 = \frac{1}{2}$ are two asymptotes of $K_2$. Red line represents the CG cut of $K_2$, derived from the blue dashed line.}
  \label{fig:sub2}
\end{subfigure}
\caption{Figures (a) and (b) demonstrate two congruent closed convex sets, whose integer hulls are both rational polyhedral, while their CG closures have completely different properties.}
\end{figure}
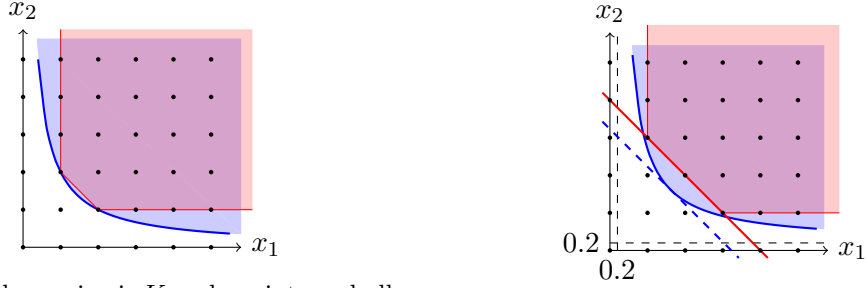

Henceforth, we will only consider exposed faces of closed convex sets, and for the sake of brevity we refer to them as \emph{faces}. In other words, a face $F$ of a closed convex set $K$ is a subset of the form $F = \{x \in K \mid \pi x = \sigma_K(\pi)\}$ for some supporting hyperplane $\pi x = \sigma_K(\pi)$. We will call the face $F$ the \emph{$\pi$-face} of $K$. Detailed definitions and properties of faces can be found in \cite{schrijver1998theory} and \cite{barvinok2002course}.

For any vector $\pi \in \R^n$, we can associate an unique rational linear subspace $V_\pi$ with it.

\begin{definition}
\label{defn: subspace_pi}
Given $\pi \in \R^n$, define 
$V_\pi := \{x \in \R^n \mid \alpha^T x = 0 \text{ for any } \alpha \in \Q^n \text{ such that } \alpha^T \pi \in \Q\}.$
\end{definition}

\begin{lemma}
\label{lem: pi_subspace}
Given a rational linear subspace $L \subseteq \R^n$, and $\pi \in L$. Then $V_\pi \subseteq L$.
\end{lemma}

\begin{proof}
Let $L: = \{x \in \R^n \mid Bx = 0\}$, where $B \in \Q^{k \times n}$. Then $B_\ell^T \pi = 0 \in \Q$, here $B_\ell$ is the $\ell$-th row of $B$. By Definition~\ref{defn: subspace_pi}, we know $V_\pi \subseteq \{x \in \R^n \mid B_\ell^T x = 0 \ \forall \ell \in [k]\} = L$.
 \end{proof}

For the rational linear subspace $V_\pi$ associated with any vector $\pi \in \R^n$, we have the following well-known simultaneous diophantine approximation theorem which is due to Kronecker \cite{kronecker1884naherungsweise}. Note that the version we used here is similar to the one used by \cite{braun2014short}. We include its proof in Appendix~\ref{append: kroneck_lemma}.

\begin{lemma}[\cite{kronecker1884naherungsweise,weyl1916gleichverteilung,braun2014short}]
\label{lem: kronecker}
Let $n, N_0 \in \N$ and $\pi \in \R^n$ with $\pi \neq 0$. Then $\Z^n - \pi \Z_{> N_0}$ contains a dense subset of $V_\pi$. 
\end{lemma}

We will also require the following classic result about the sensitivity of Linear Programming.

\begin{lemma}[Sticky face lemma \cite{robinson2018short}]
\label{lem: sticky}
If $P$ is a polyhedron in $\R^n$, $x_0^*$ is a point of $\R^n$ and $F$ is the set of maximizers of $\langle x^*_0, \cdot \rangle$ on $P$ (a face of $P$). Then for any $x^*$ close enough to $x^*_0$, the maximizers of $\langle x^*, \cdot \rangle$ on $P$ are just its maximizers on $F$.
\end{lemma}

\subsection{Sufficient Condition}

In this section, we want to establish the sufficient condition in Theorem~\ref{theo: motzkin_equiv}, for finitely-generated property of the CG closure.
Using the main Theorem~\ref{theo: main_CG}, it suffices for us to show the following result.

\begin{proposition}
\label{prop: rational_recession_cone_thin}
If $K$ is a Motzkin-decomposable set in $\R^n$ with rational polyhedral recession cone, then
there exists a finite subset $F \subseteq \Z^n$, such that $\left\{x \in \R^n \mid f x \leq \lfloor \sigma_K(f) \rfloor, \forall f \in F \right\} \subseteq K.$
\end{proposition}

Before presenting the proof of Proposition~\ref{prop: rational_recession_cone_thin}, we will need the following auxiliary results. The first lemma can be viewed as an extension of the \emph{continuity Lemma~1} in \cite{braun2014short} and \emph{sticky face lemma}~\ref{lem: sticky}. Note that unlike Proposition~\ref{prop: rational_recession_cone_thin}, here we do not assume \textbf{rational} polyhedral recession cone. 

\begin{lemma}
\label{lem: sticky_generalization}
Let $K$ be a Motzkin-decomposable set with polyhedral recession cone, and $F$ is a $\pi$-face of $K$. For any $\delta >0,$ let
$F_\delta: = \{x \in K \mid \exists \ x' \in F \text{ s.t. } \|x-x'\| \leq \delta\}.$
Then there exists $\epsilon > 0$, such that for any $\pi'$ with $\|\pi' - \pi\| < \epsilon, \sigma_K(\pi') = \sigma_{F_\delta}(\pi')$.
\end{lemma}

\begin{proof}
By assumption, since $K$ is Motzkin-decomposable with polyhedral recession cone, then we write $K = C + \cone(R)$ for a compact convex set $C$ and a finite set of extreme rays $R$. 
Let $F = \{x \in K \mid \pi x = \pi_0\}$ be the $\pi$-face of $K$ and $\pi x = \pi_0$ is a supporting hyperplane ($\pi_0 < \infty$), we know that $\pi r \leq 0$ for any $r \in R$, and $R_0: = \{r \in R \mid \pi r = 0\}$ is the set of extreme rays of $F$. 
Clearly $R_0$ is also the set of extreme rays of $F_\delta$.

We prove the statement of this lemma by contradiction: there exists a convergent sequence $\pi^i \rightarrow \pi$ and $\sigma_K(\pi^i) > \sigma_{F_\delta}(\pi^i)$. 
Note that here $\sigma_{F_\delta}(\pi^i) < \infty$, which implies $\pi^i r \leq 0$ for any $r \in R_0$. By definition of $R_0$, we know that for any $r \in R \setminus R_0$, there is $\pi r < 0$. 
Here $R \setminus R_0$ is a finite set,
hence for any $\pi^i$ close enough to $\pi$, we also have $\pi^i r < 0$ for any $r \in R \setminus R_0$. Therefore, for any $\pi^i$ close enough to $\pi$, there is $\pi^i r \leq 0$ for any $r \in R$. W.l.o.g., we can assume that for our sequence $\{\pi^i\}_{i \geq 1}, \pi^i r \leq 0$ for any $r \in R$ and $i \geq 1$. 
From $\sigma_K(\pi^i) > \sigma_{F_\delta}(\pi^i)$ for any $i \geq 1$, we know there must exist $x^i \in K \setminus F_\delta$, such that $\pi^i x^i > \sigma_{F_\delta}(\pi^i) = \max_{x \in F_\delta \cap C} \pi^i x$. From our above assumption that $\pi^i r \leq 0$ for any $r \in R$, here we can further assume that $x^i \in C \setminus F_\delta$.
Since $x^i \in C$ which is a compact set, by Bolzano-Weierstrass theorem, there is a convergent subsequence of $\{x^i\}_{i \geq 1}$. W.l.o.g. we still assume the convergent subsequence of $\{x^i\}_{i \geq 1}$ is itself, and $x^i \rightarrow x^* \in C$. Note that $x^i \notin F_\delta$, so we have $x^* \notin F$. Therefore, from $\pi^i x^i > \max_{x \in F_\delta \cap C} \pi^i x$, we have:
\begin{equation*}
\pi x^* = \lim_{i \rightarrow \infty} \pi^i x^i \geq \lim_{i \rightarrow \infty} \max_{x \in F_\delta \cap C} \pi^i x = \max_{x \in F_\delta \cap C} \pi x = \pi_0.
\end{equation*}
Since $\pi x = \pi_0$ is a supporting hyperplane of $K$, we obtain that $\pi x^* = \pi_0$ and $x^* \in F$, which gives the contradiction.
 \end{proof}

Next we present the key lemma for establishing the proof of Proposition~\ref{prop: rational_recession_cone_thin}.

\begin{lemma}
\label{lem: final_key_lem}
Let $K$ be a Motzkin-decomposable set with rational polyhedral recession cone. 
For any $\pi$-face $F$ of $K$, if $F'$ is finitely-generated, then there exists a rational polyhedron $P_\pi$ obtained from finitely many CG cuts of $K$ and $\epsilon_\pi > 0$, such that for any $\pi'$ with $\|\pi' - \pi\| < \epsilon_\pi, \pi' x \leq \sigma_K(\pi')$ is valid to $P_\pi$. 
\end{lemma}

\begin{proof}
Let $K = C + \cone(R)$, where $C$ is a compact convex set and $R$ is a finite set of rational extreme rays of $K$. Denote $u: = \max_{x \in C}\|x\| < \infty$. 
Here we can find a multiplier $\kappa>0$, such that $(\alpha, \alpha_0) := \kappa(\pi, \sigma_K(\pi))$ with $\alpha_0 \in \Z$. In the following discussion, we simply denote the supporting hyperplane $\pi x= \sigma_K(\pi)$ of $K$ as $\alpha x = \alpha_0$.

By assumption that $F'$ is finitely-generated, we can denote $F' = \{x \in \R^n \mid g x \leq \lfloor \sigma_F(g) \rfloor, \forall g \in G\}$, with $g \in \Z^n, \forall g \in G$. Here w.l.o.g. we assume that $0 \in G$, since $0 x \leq \lfloor \sigma_F(0) \rfloor$ trivially holds.

Pick a small positive number $$\delta < \min_{g \in G} \frac{1 + \lfloor \sigma_F(g)\rfloor  - \sigma_{F}(g)}{2},$$ and choose a neighborhood of $F$: $$\bar F: = \{x \in K \mid \exists \ x' \in F \text{ s.t. } \|x-x'\| \leq \frac{\delta}{\max_{g \in G} \|g\|}\}.$$
By Lemma~\ref{lem: sticky_generalization}, we know there exists a positive number $\epsilon_0 > 0$, such that for any $\alpha'$ with $\|\alpha' - \alpha\| < \epsilon_0$, there is $\sigma_K(\alpha') = \sigma_{\bar F}(\alpha')$. Furthermore, there is a large enough integer number $N_0$, such that for any positive integer $m \geq N_0$ and any vector $c$ with $\|c-m \alpha\| \leq \frac{\delta}{u}$, we have $\|\frac{c+g}{\|c+g\|} - \frac{\alpha}{\|\alpha\|}\| \leq \frac{\epsilon_0}{\|\alpha\|}$ for any $g \in G$, which is $\|(c+g)\frac{\|\alpha\|}{\|c+g\|} - \alpha \| \leq \epsilon_0$. 

By Lemma~\ref{lem: kronecker}, $\Z^n - \alpha \Z_{>N_0}$ contains a dense subset of $V_\alpha$, so we can find some $c^i - m_i \alpha, \lambda_i \in [0,1]$ for $i \in [k]$ with $\sum_{i=1}^k \lambda_i = 1$, such that 
\begin{equation}
\label{eq: dense_subspace_simplex}
\sum_{i\in [k]} \lambda_i (c^i - m_i \alpha) = 0, \|c^i - m_i \alpha\| \leq \frac{\delta}{u},  \quad c^i \in \Z^n, m_i \in \N_{>N_0}, c^i - m_i \alpha \in V_\alpha.
\end{equation}

\begin{claim}
For any $x \in \bar F$ and $v \in V_\alpha, vx \leq \|v\|u$.
\end{claim}
\begin{cpf}
Let $R_0: = \{r \in R \mid \alpha r = 0\}$ denote the set of extreme rays of the face $F$. Then obviously $R_0$ is also the set of extreme rays of the set $\bar F$. 
Since $\alpha \in \ker(R_0)$, where $\ker(R_0)$ is a rational linear subspace because $R$ is assumed to be a finite set of rational vectors, hence from Lemma~\ref{lem: pi_subspace}, we know $V_{\alpha} \subseteq \ker(R_0)$.
This implies that, for any $v \in V_\alpha$ and $x = y + r \in C+R_0$, there is $vx = v y \leq \|v\|u$.
\end{cpf}
Therefore, for any $x \in \bar F, i \in [k]$ and $g \in G$, we have
\begin{align}
\label{eq: sigma_F}
\begin{split}
(c^i + g) x & = gx + m_i \alpha x + (c^i-m_i \alpha)x \\
& \leq \sigma_F(g) + \delta + m_i \alpha_0 + \delta.
\end{split}
\end{align}
Here for any $x \in \bar F, gx \leq \sigma_F(g) + \delta$ is from the definition of $\bar F$, and $(c^i-m_i \alpha)x \leq \delta$ is from the last claim and the fact that $c^i - m_i \alpha \in V_\alpha$ and $\|c^i - m_i \alpha\| \leq \frac{\delta}{u}$.
According to our construction of $N_0$ and $\epsilon_0$, we know that $\sigma_{K}((c^i+g)\frac{\|\alpha\|}{\|c^i+g\|}) = \sigma_{\bar F}((c^i+g)\frac{\|\alpha\|}{\|c^i+g\|})$, for any $i \in [k]$. Hence:
$$\sigma_K(c^i + g)  = \sigma_{\bar F}(c^i+g) \leq \sigma_F(g)  + m_i \alpha_0 + 2\delta.$$
Here the last inequality is from \eqref{eq: sigma_F}.
This implies that
$$(c^i + g) x  \leq \lfloor \sigma_F(g)  + m_i \alpha_0 + 2\delta \rfloor  = \lfloor \sigma_F(g) \rfloor + m_i \alpha_0$$
is a CG cut of $K$, for any $i \in [k]$ and $g \in G$. Here the last equality is because $\delta$ is assumed to be less than $\min_{g \in G} \frac{1+\lfloor \sigma_{F}(g) \rfloor - \sigma_{F}(g)}{2}$, and $m_i , \alpha_0 \in \Z$. 
Now we denote 
\begin{equation}
\label{eq: P_pi_defn}
P_\pi: = \{x \in \R^n \mid (c^i + g) x \leq \lfloor \sigma_F(g) \rfloor + m_i \alpha_0, \forall i \in [k], g \in G\}.
\end{equation}
Here $P_\pi$ is a rational polyhedron which is obtained from finitely-many CG cuts of $K$. 
\begin{claim}
\label{claim: valid_and_intersection_subset}
$\alpha x \leq \alpha_0$ is valid to $P_\pi$, and $\{x \in P_\pi \mid \alpha x = \alpha_0\} \subseteq F$.
\end{claim}

\begin{cpf}
For any $g \in G$, by definition of $P_\pi$, we know that inequality $(\sum_{i \in [k]} \lambda_i c^i + g )x \leq \lfloor \sigma_F(g) \rfloor + \sum_{i \in [k]} \lambda_i m_i \alpha_0$ is valid to $P_\pi$. By \eqref{eq: dense_subspace_simplex}, such inequality is just $(\gamma \alpha + g)x \leq \lfloor \sigma_F(g) \rfloor + \gamma \alpha_0$, where we denote $\gamma : = \sum_{i \in [k]} \lambda_i m_i$.
By assumption that $0 \in G$, we know $\alpha x \leq \alpha_0$ is valid to $P_\pi$.
Moreover, there is 
\begin{align*}
\{x \in P_\pi \mid \alpha x = \alpha_0\}  & \subseteq \{x \in \R^n \mid \alpha x = \alpha_0, (\gamma \alpha + g)x \leq \lfloor \sigma_F(g) \rfloor + \gamma \alpha_0, \forall g \in G\}\\
& \subseteq \{x \in \R^n \mid g x \leq \lfloor \sigma_F(g) \rfloor, \forall g \in G\} \\
& = F' \subseteq F.
\end{align*}
Hence this claim holds.
\end{cpf}
Lastly, we want to show that, for $\alpha'$ close enough to $\alpha, \alpha' x \leq \sigma_K(\alpha')$ will always be valid to $P_\pi$. Denote $F_\pi$ to be the $\alpha$-face of $P_\pi$, and let $\sigma: = \alpha x$ for any arbitrary $x \in F_\pi$. Here $F_\pi = \{x \in P_\pi \mid \alpha x = \sigma\}$.
By the sticky face lemma~\ref{lem: sticky}, for polyhedron $P_\pi$, we know there exists $\epsilon_1>0$, such that when $\|\alpha' - \alpha\| < \epsilon_1$, $\sigma_{P_\pi}(\alpha') = \sigma_{F_\pi}(\alpha')$.
It suffices for us to show, for $\alpha'$ close enough to $\alpha$ there is $\sigma_{F_\pi}(\alpha') \leq \sigma_K(\alpha')$, because this will imply that $\sigma_{P_\pi}(\alpha') \leq \sigma_K(\alpha')$, meaning $\alpha' x \leq \sigma_K(\alpha')$ is also valid to $P_\pi$.
Note that Claim~\ref{claim: valid_and_intersection_subset} tells that $\alpha x \leq \alpha_0$ is valid to $P_\pi$, we have $\sigma \leq \alpha_0$.
Next we argue by two cases:
\begin{enumerate}
\item \emph{Case $\sigma = \alpha_0$:} 
In this case, $F_\pi = \{x \in P_\pi \mid \alpha x = \alpha_0\}$.
By Claim~\ref{claim: valid_and_intersection_subset}, there is $F_\pi \subseteq F \subseteq K$, which implies that $\sigma_{F_\pi} (\alpha') \leq \sigma_K(\alpha')$. Hence in this case, when $\alpha'$ is close enough to $\alpha, \sigma_{P_\pi}(\alpha') = \sigma_{F_\pi}(\alpha') \leq \sigma_K(\alpha')$, completing the proof.
 \item \emph{Case $\sigma < \alpha_0$:} We decompose the face $F_\pi$ as: $F_\pi = \conv(E_\pi) + \cone(R_\pi)$, where $E_\pi$ and $R_\pi$ denote the set of extreme points and extreme rays of $F_\pi$ respectively. 
Arbitrarily pick $r \in R_\pi$, it is also an extreme ray of polyhedron $P_\pi$. By the definition~\eqref{eq: P_pi_defn} of $P_\pi$, we know $(c^i + g) r \leq 0$ for any $i \in [k]$ and $g \in G$. Since $r$ is an extreme ray of face $F_\pi = \{x \in P_\pi \mid \alpha x = \sigma\}$, we also have $\alpha r = 0$. By \eqref{eq: dense_subspace_simplex}, we know there must exist some $i' \in [k]$, such that $c^{i'} r \geq 0$. Combined with the fact that $(c^{i'} + g) r \leq 0$ for any $g \in G$, we have: $g r \leq 0, \ \forall g \in G$. Recall that $F' = \{x \in \R^n \mid gx \leq \lfloor \sigma_{F}(g) \rfloor, \forall g \in G\}$, hence $r$ is also a ray of $F'$, which is contained in $K$. This implies that, for any $\alpha'$, if $\sigma_{F_\pi}(\alpha') = \infty$, then $\sigma_K(\alpha') = \infty$. Therefore, we only have to show, for any $\alpha'$ close enough to $\alpha$ and $\sigma_{F_\pi}(\alpha') < \infty$, then $\alpha' x \leq \sigma_K(\alpha')$ is valid to $P_\pi$. For any $\alpha'$ close enough to $\alpha$ with $\sigma_{F_\pi}(\alpha') < \infty$, there is 
$$
\sigma_{F_\pi}(\alpha')  = \max_{x \in E_\pi} \alpha' x  \leq \max_{x \in E_\pi} \alpha x + \frac{\alpha_0 - \sigma}{2} = \frac{\sigma+\alpha_0}{2}.
$$
Here the second inequality is because, $\sigma_{E_\pi} (\alpha') \rightarrow \sigma_{E_\pi} (\alpha)$ as  $\alpha' \rightarrow \alpha$.
Moreover, arbitrarily pick a point $x^* \in F$, when $\|\alpha'-\alpha\| < \frac{\alpha_0 - \sigma}{2 \|x^*\|}$, there is 
$$
\alpha' x^* = \alpha x^* + (\alpha' - \alpha) x^* \geq  \alpha_0 - \|\alpha'-\alpha\| \cdot \|x^*\|  > \frac{\sigma+\alpha_0}{2}.
$$
Hence, $\sigma_K(\alpha') > \frac{\sigma+\alpha_0}{2} \geq \sigma_{F_\pi}(\alpha')$ when $\alpha'$ is sufficiently close to $\alpha$.
This concludes the proof for this case.
\end{enumerate}
\smallskip
Therefore, we have shown that, there exists a small constant $\epsilon_\alpha > 0$, such that for any $\alpha'$ with $\|\alpha' - \alpha \| < \epsilon_\alpha, \alpha' x \leq \sigma_K(\alpha')$ is always valid to $P_\pi$. Note that $\alpha = \kappa \pi$ and $\sigma_K(\kappa \pi') = \kappa \sigma_K(\pi')$, by picking $\epsilon_\pi: = \frac{\epsilon_\alpha}{\kappa}$, we conclude the proof.
 \end{proof}

Now we have all the tools needed to verify Proposition~\ref{prop: rational_recession_cone_thin}. 

\begin{proof}[Proof of Proposition~\ref{prop: rational_recession_cone_thin}]
Denote $K = C + \cone(R)$, where $C$ is a compact convex set and $R$ is a finite set of rational extreme rays of $K$. 
The proof proceeds via induction on the dimension of $K$. By inductive hypothesis, this proposition holds for proper faces of $K$. Further from Theorem~\ref{theo: main_CG}, we know that the CG closure of any proper face of $K$ is finitely-generated. Let 
$$
\Pi: = \{\pi \in \R^n \mid \|\pi\| = 1, \pi r \leq 0 \ \forall r \in R\}.
$$
Since a closed convex set can be exactly given by intersecting all of its supporting half-spaces, there is $K = \{x \in \R^n \mid \pi x \leq \sigma_K(\pi), \forall \pi \in \Pi\}$. For any $\pi \in \Pi$, since the CG closure of any proper face of $K$ is finitely-generated, by Lemma~\ref{lem: final_key_lem}, we know there exists a rational polyhedron $P_\pi$ obtained from finitely many CG cuts of $K$ and a positive number $\epsilon_\pi$, such that for any $\pi'$ with $\|\pi' - \pi\| < \epsilon_\pi, \pi' x \leq \sigma_K(\pi')$ will be valid to $P_\pi$. Hence we obtain an open cover $\{O_{\epsilon_\pi}(\pi)\}_{\pi \in \Pi}$ for $\Pi$. Because $\Pi$ is a compact set, as a consequence, there exists a finite subset $\bar{\Pi} \subseteq \Pi$ with $\Pi \subseteq \{O_{\epsilon_\pi}(\pi)\}_{\pi \in \bar{\Pi}}$. Consider polyhedron 
\begin{equation}
\label{eq: PPP}
P: = \bigcap_{\pi \in \bar{\Pi}} P_\pi.
\end{equation}
Here we know that $P$ is also obtained from finitely many CG cuts of $K$. Moreover, for any supporting half-space $\pi' x \leq \sigma_K(\pi')$ of $K$, since $\pi' \in \Pi \subseteq \{O_{\epsilon_\pi}(\pi)\}_{\pi \in \bar{\Pi}}$, we know there exists $\pi'' \in \bar{\Pi}$ such that $\pi' \in O_{\epsilon_{\pi''}}(\pi'')$. Hence $\pi' x \leq \sigma_K(\pi')$ is valid to $P_{\pi''}$, which contains $P$. In other words, we have shown that, for any supporting half-space of $K$, this half-space also contains $P$. Therefore, our constructed $P$ in \eqref{eq: PPP} is contained in $K$, and we complete the proof.
 \end{proof}

\subsection{Proof of Theorem~\ref{theo: motzkin_equiv}}

To show the necessary condition for the rational polyhedrality of $K'$ in Theorem~\ref{theo: motzkin_equiv}, we will need the following easy lemma.
\begin{lemma}
\label{lem: same_recess}
For a closed convex set $K$, $\rec(K) = \rec(K')$.
\end{lemma}

\begin{proof}
Since $K' \subseteq K$, it suffices to show: $\rec(K) \subseteq \rec(K')$. Arbitrarily pick $r \in \rec(K)$, 
denote $C_r: = \{c \in \Z^n \mid c \cdot r \leq 0\}$. For any $c \in \Z^n \setminus C_r$, since $c \cdot r > 0$, we know $\sigma_K(c) = \infty$. Therefore,
$
K' = \bigcap_{c \in C_r} \{x \in \R^n \mid cx \leq \lfloor \sigma_K(c) \rfloor\}.
$
Since $c \cdot r \leq 0$ for any $c \in C_r$, we obtain $r \in \rec(K')$. By the arbitrariness of $r \in \rec(K)$, we conclude the proof of $\rec(K) \subseteq \rec(K')$.
\end{proof}

Now we are ready to prove Theorem~\ref{theo: motzkin_equiv}.

\begin{proof}[Proof of Theorem~\ref{theo: motzkin_equiv}]
Let $K$ be a Motzkin-decomposable set. First, assume $\rec(K)$ is a rational polyhedral cone. Then by Proposition~\ref{prop: rational_recession_cone_thin} and Theorem~\ref{theo: main_CG}, we know that $K'$ is 
finitely-generated. Now assume that $K'$ is a finitely-generated. Then $K'$ is also a rational polyhedron, which has rational polyhedral recession cone. By Lemma~\ref{lem: same_recess}, we obtain that $K$ also has rational polyhedral recession cone.
\end{proof}

When $K$ is further assumed to contain integer points in its interior, we have the following necessary condition for $\conv(K \cap \Z^n)$ to be a rational polyhedron. 

\begin{proposition}[Theorem~6 \cite{MR3097296}]
\label{prop: integerhull_poly_necessary}
Let $K$ be a closed convex set in $\R^n$. If $\inter(K) \cap \Z^n \neq \emptyset$ and $\conv(K \cap \Z^n)$ is a polyhedron, then $\rec(K)$ is a rational polyhedral cone. 
\end{proposition}

From the last proposition and Theorem~\ref{theo: motzkin_equiv}, we obtain Corollary~\ref{cor: main} as an immediate corollary.

\begin{proof}[Proof of Corollary~\ref{cor: main}]
Let $K$ be a Motzkin-decomposable set which contains integer points in its interior. First, assume $\conv(K \cap \Z^n)$ is a polyhedron. Then by Proposition~\ref{prop: integerhull_poly_necessary}, we know that $K$ has rational polyhedral recession cone. 
From Theorem~\ref{theo: motzkin_equiv}, we obtain that $K'$ is a rational polyhedron. Now, assuming $K'$ is a rational polyhedron. By the fact that $K' \cap \Z^n = K \cap \Z^n$, we know $\conv(K \cap \Z^n) = \conv(K' \cap \Z^n)$, which is a rational polyhedron.
\end{proof}

For closed convex sets which are not Motzkin-decomposable,
as we have seen from Example~\ref{exam: hyper1} and Example~\ref{exam: hyper2}, the integer hull of $K_1$ and $K_2$ are both polyhedral, while $K'_1$ is non-polyhedral, and $K'_2$ is polyhedral. The fact that $K_1$ is congruent with $K_2$ suggests that for more general closed convex set, the relationship between its integer hull and its CG closure is more subtle. Moreover, we should further remark that, the additional condition that $\inter(K) \cap \Z^n \neq \emptyset$ is not artificial. 

\begin{example}
Let $K = \{x \in \R^2 \mid \sqrt{2}x_1 - x_2 = 0\}$, which is a straight line with irrational slope. Then $K \cap \Z^2 = \{0\}$, and its integer hull is a singleton (also a polyhedron). However, $K' = K$ which is an irrational polyhedron.
\end{example}

\bibliographystyle{plain}
\bibliography{cite}

\begin{thebibliography}{10}

\bibitem{MR2969261}
Gennadiy Averkov.
\newblock On finitely generated closures in the theory of cutting planes.
\newblock {\em Discrete Optim.}, 9(4):209--215, 2012.

\bibitem{barvinok2002course}
Alexander Barvinok.
\newblock {\em A course in convexity}, volume~54.
\newblock American Mathematical Soc., 2002.

\bibitem{bockmayr1999chvatal}
Alexander Bockmayr, Friedrich Eisenbrand, Mark Hartmann, and Andreas~S Schulz.
\newblock On the chv{\'a}tal rank of polytopes in the 0/1 cube.
\newblock {\em Discrete Applied Mathematics}, 98(1-2):21--27, 1999.

\bibitem{braun2014short}
G{\'a}bor Braun and Sebastian Pokutta.
\newblock A short proof for the polyhedrality of the chv{\'a}tal--gomory
  closure of a compact convex set.
\newblock {\em Operations Research Letters}, 42(5):307--310, 2014.

\bibitem{CHVATAL1973305}
V\'aclav Chv\'atal.
\newblock Edmonds polytopes and a hierarchy of combinatorial problems.
\newblock {\em Discrete Mathematics}, 4(4):305--337, 1973.

\bibitem{MR986890}
V\'aclav Chv\'{a}tal, William Cook, and Mark Hartmann.
\newblock On cutting-plane proofs in combinatorial optimization.
\newblock {\em Linear Algebra Appl.}, 114/115:455--499, 1989.

\bibitem{cook1986integer}
William Cook, Jean Fonlupt, and Alexander Schrijver.
\newblock An integer analogue of caratheodory's theorem.
\newblock {\em Journal of Combinatorial Theory, Series B}, 40(1):63--70, 1986.

\bibitem{dadush2011chvatal}
Daniel Dadush, Santanu~S Dey, and Juan~Pablo Vielma.
\newblock The chv{\'a}tal-gomory closure of a strictly convex body.
\newblock {\em Mathematics of Operations Research}, 36(2):227--239, 2011.

\bibitem{dadush2014chvatal}
Daniel Dadush, Santanu~S Dey, and Juan~Pablo Vielma.
\newblock On the chv{\'a}tal--gomory closure of a compact convex set.
\newblock {\em Mathematical Programming}, 145(1-2):327--348, 2014.

\bibitem{MR4207341}
Alberto Del~Pia, Dion Gijswijt, Jeff Linderoth, and Haoran Zhu.
\newblock Integer packing sets form a well-quasi-ordering.
\newblock {\em Oper. Res. Lett.}, 49(2):226--230, 2021.

\bibitem{MR3097296}
Santanu~S Dey and Diego~A Mor\'{a}n~R.
\newblock Some properties of convex hulls of integer points contained in
  general convex sets.
\newblock {\em Math. Program.}, 141(1-2, Ser. A):507--526, 2013.

\bibitem{dey2010chvatal}
Santanu~S Dey and Juan~Pablo Vielma.
\newblock The chv{\'a}tal-gomory closure of an ellipsoid is a polyhedron.
\newblock In {\em International Conference on Integer Programming and
  Combinatorial Optimization}, pages 327--340. Springer, 2010.

\bibitem{dickson1913finiteness}
Leonard~Eugene Dickson.
\newblock Finiteness of the odd perfect and primitive abundant numbers with n
  distinct prime factors.
\newblock {\em American Journal of Mathematics}, 35(4):413--422, 1913.

\bibitem{dunkel2013gomory}
Juliane Dunkel and Andreas~S Schulz.
\newblock The gomory-chv{\'a}tal closure of a nonrational polytope is a
  rational polytope.
\newblock {\em Mathematics of Operations Research}, 38(1):63--91, 2013.

\bibitem{MR2306128}
Matteo Fischetti and Andrea Lodi.
\newblock Optimizing over the first {C}hv\'{a}tal closure.
\newblock {\em Math. Program.}, 110(1, Ser. B):3--20, 2007.

\bibitem{goberna1998linear}
Miguel~A Goberna.
\newblock Linear semi-infinite optimization.
\newblock {\em Mathematical Methods in Practice 2}, 1998.

\bibitem{goberna2010motzkin}
Miguel~A Goberna, Enrique Gonz{\'a}lez, Juan~Enrique Mart{\'\i}nez-Legaz, and
  Maxim~I Todorov.
\newblock Motzkin decomposition of closed convex sets.
\newblock {\em Journal of mathematical analysis and applications},
  364(1):209--221, 2010.

\bibitem{MR102437}
Ralph~E Gomory.
\newblock Outline of an algorithm for integer solutions to linear programs.
\newblock {\em Bull. Amer. Math. Soc.}, 64:275--278, 1958.

\bibitem{husseinov1999note}
Farhad H{\"u}sseinov.
\newblock A note on the closedness of the convex hull and its applications.
\newblock {\em Journal of Convex Analysis}, 6(2):387--393, 1999.

\bibitem{iusem2014motzkin}
Alfredo~N Iusem, Juan~Enrique Martinez-Legaz, and Maxim~I Todorov.
\newblock Motzkin predecomposable sets.
\newblock {\em Journal of Global Optimization}, 60(4):635--647, 2014.

\bibitem{klee1957extremal}
Victor Klee.
\newblock Extremal structure of convex sets.
\newblock {\em Archiv der Mathematik}, 8(3):234--240, 1957.

\bibitem{kronecker1884naherungsweise}
Leopold Kronecker.
\newblock {\em N{\"a}herungsweise ganzzahlige aufl{\"o}sung linearer
  gleichungen}.
\newblock 1884.

\bibitem{robinson2018short}
Stephen~M Robinson.
\newblock A short proof of the sticky face lemma.
\newblock {\em Mathematical Programming}, 168(1-2):5--9, 2018.

\bibitem{schrijver1980cutting}
Alexander Schrijver.
\newblock On cutting planes.
\newblock {\em Combinatorics}, 79:291--296, 1980.

\bibitem{schrijver1998theory}
Alexander Schrijver.
\newblock {\em Theory of linear and integer programming}.
\newblock John Wiley \& Sons, 1998.

\bibitem{studeny1993convex}
Milan Studen{\`y}.
\newblock Convex cones in finite-dimensional real vector spaces.
\newblock {\em Kybernetika}, 29(2):180--200, 1993.

\bibitem{weyl1916gleichverteilung}
Hermann Weyl.
\newblock {\"U}ber die gleichverteilung von zahlen mod. eins.
\newblock {\em Mathematische Annalen}, 77(3):313--352, 1916.

\bibitem{zhu2021characterization}
Haoran Zhu.
\newblock Characterization of the cutting-plane closure, 2021.
\newblock \url {https://arxiv.org/abs/1911.12943}.

\end{thebibliography}

\appendix
\addcontentsline{toc}{section}{Appendices}
\renewcommand{\thesubsection}{\Alph{subsection}}

\section*{Appendices}
\label{append: sec_prelim}

\subsection{Proof of Proposition~\ref{prop: valid_ineq_for_closure} and Corollaries}
\label{append: 1}

We will require the next extended Farkas' lemma for the proof of Proposition~\ref{prop: valid_ineq_for_closure}.
\begin{lemma}[Extended Farkas' lemma, Corollary 3.1.2 \cite{goberna1998linear}]
\label{lem: extended_Farkas}
The inequality $ax \geq b$ is a consequence of the consistent system $\{a_t x \geq b_t, t \in T\}$ if and only if $(a,b) \in \cl \cone(\{(a_t, b_t) \ \forall t \in T, (0, \ldots, 0, -1)\})$.
\end{lemma}

\begin{proof}[Proof of Proposition~\ref{prop: valid_ineq_for_closure}]
Consider the linear system $\{\omega \cdot (x,-1) \ \forall \omega \in \Omega\}$. Since $\I(\Omega)$ is essentially the feasible region given by this linear system which is also non-empty, we know that linear system $\{\omega \cdot (x,-1) \ \forall \omega \in \Omega\}$ is consistent. By extended Farkas' lemma~\ref{lem: extended_Farkas} and the assumption that $(0, \ldots, 0, 1) \in \Omega$, we obtain $\alpha x \leq \beta$ is valid to $\I(\Omega)$ if and only if $(\alpha, \beta) \in \cl \cone(\Omega)$. 
\end{proof}

\begin{proof}[Proof of Corollary~\ref{cor: finitely_generated}]
By definition, $\I(\Omega)$ is finitely-generated if there exists a finite subset $\bar \Omega \subseteq \Omega$ with $\I(\bar \Omega) = \I(\Omega)$. It suffices to show: for any finite subset $\bar \Omega \subseteq \Omega$, $\I(\bar \Omega) = \I(\Omega)$ if and only if $\cone(\bar \Omega) = \cl \cone(\Omega)$.
Note that $\I(\bar \Omega) = \I(\Omega)$ is equivalent of saying: any inequality $\alpha x \leq \beta$ is valid to $\I(\bar \Omega)$ if and only if it is also valid to $\I(\Omega)$.
By Proposition~\ref{prop: valid_ineq_for_closure}, that is further equivalent of saying: $(\alpha, \beta) \in \cone(\bar \Omega)$ if and only if $(\alpha, \beta) \in \cl \cone(\Omega)$. Thus we complete the proof.
\end{proof}

%

\subsection{Proof of Lemma~\ref{lem: characterization_extremeray}}
\label{append: extremeray_char}

First, we present some well-known results in convex geometry that will be needed. 
\begin{lemma}[Supporting Hyperplane Theorem for pointed cone]
\label{lem: supporting_hyperplane_pointed}
Let $K \subseteq \R^n$ be a closed convex pointed cone. Then there is $h \in \R^n$ such that if $x \in K$ and $x \neq 0$, then $h^T x > 0$.
\end{lemma}
\begin{proof}
Since $K$ is pointed, we know the polar cone $K^\circ$ is full-dimensional. So we can find an interior point $x^* \in K^\circ$, which has $x^* \cdot x < 0$ for all $x \in K$. By picking $h = -x^*$ we complete the proof.
\end{proof}

\begin{lemma}[Lemma 2.4 in \cite{husseinov1999note}, Theorem 3.5 \cite{klee1957extremal}]
\label{lem: victor_klee_extreme}
Let $S$ be a non-empty closed set in $\R^n$. Then, every extreme point of $\cl \conv(S)$ belongs to $S$.
\end{lemma}

Now we are ready to verify Lemma~\ref{lem: characterization_extremeray}.
\begin{proof}[Proof of Lemma~\ref{lem: characterization_extremeray}]
From Lemma~\ref{lem: supporting_hyperplane_pointed}, we can find a supporting hyperplane $h x = 0$ such that $h \omega > 0$ for all $\omega \neq 0 \in \cl \cone(\Omega)$. 
 Denote the normalized version of $\Omega$: $\Omega' = \{\frac{\omega}{h \cdot \omega} \mid \omega \in \Omega\}$, which is well-defined since $0 \notin \Omega$, and for all $\omega \neq 0 \in \Omega$ there is $h \cdot \omega > 0$. 
 \begin{claim}
 $\{x \in \R^n \mid h x  = 1\} \cap \cl \cone(\Omega) = \cl \conv (\Omega')$.
 \end{claim}
 \begin{cpf}
 First, we show  $\{x \in \R^n \mid h x  = 1\} \cap \cl \cone(\Omega) \subseteq \cl \conv (\Omega')$. Arbitrarily pick $\alpha^*$ such that $h  \alpha^* = 1$, and there exists $\{\alpha^i\} \subseteq \cone(\Omega)$ such that $\alpha^i \rightarrow \alpha^*$. Denote $\beta^i: = \frac{\alpha^i}{h  \alpha^i}$. Since $\alpha^i \rightarrow \alpha^*, h  \alpha^* = 1$, we know $h  \alpha^i \rightarrow 1$. Hence we also have $\beta^i \rightarrow \alpha^*$, and here $\beta^i \in \{x \in \R^n \mid h x = 1\} \cap \cone(\Omega)$.  In the following, we show: $ \{x \in \R^n \mid h  x = 1\} \cap \cone (\Omega) \subseteq  \conv (\Omega')$, which will imply that $\alpha^* \in \cl \conv (\Omega')$ since $\beta^i \rightarrow \alpha^*$ and $\beta^i \in \{x \in \R^n \mid h x = 1\} \cap \cone(\Omega)$. According to the arbitrariness of $\alpha^* \in \{x \in \R^n \mid h x  = 1\} \cap \cl \cone(\Omega)$, this will complete the proof of $\{x \in \R^n \mid h x  = 1\} \cap \cl \cone(\Omega) \subseteq \cl \conv (\Omega')$. 

Pick $\beta \in \{x \in \R^n \mid h  x = 1\} \cap \cone (\Omega),$ we can write it as: $\beta = \sum_{i=1}^k \lambda_i b^i$ for some $\lambda_i > 0, b^i \in \Omega, i \in [k], k \in \N$. Here because $\beta \in \{x \in \R^n \mid h  x = 1\} $, we know $\sum_{i=1}^k \lambda_i h b^i = 1$. Therefore, we can also write $\beta$ as:
$\beta = \sum_{i=1}^k ( \lambda_i h b^i) \cdot \frac{b^i}{h b^i}, \text{ here } \frac{b^i}{h b^i} \in \Omega',\ \sum_{i=1}^k \lambda_i h b^i = 1.$
We get $\beta \in \conv(\Omega')$, which concludes $ \{x \in \R^n \mid h  x = 1\} \cap \cone (\Omega) \subseteq  \conv(\Omega')$.

Lastly, we show the other direction $\{x \in \R^n \mid h  x = 1\} \cap \cl \cone(\Omega) \supseteq \cl \conv (\Omega')$. By definition, $\Omega' \subseteq \{x \in \R^n \mid h  x = 1\}$, which implies $\cl \conv (\Omega') \subseteq \{x \in \R^n \mid h  x = 1\}$. On the other hand, clearly $\Omega' \subseteq \cone(\Omega)$, so $\cl \conv (\Omega')\subseteq \cl \cone (\Omega)$, and we complete the proof for this claim. 
\end{cpf}
Given an extreme ray $r \in \cl \cone(\Omega)$, w.l.o.g. we assume $h  r = 1$.
Then $r \in \Omega$ iff $r \in \Omega'$.
From the above claim, we also know $r \in \cl \conv (\Omega')$. Lastly, we want to show that $r$ is an extreme point of $\cl \conv (\Omega')$.
Assume $r = \sum_{i=1}^k \lambda_i a^i$ for $\lambda_i > 0, \sum_{i=1}^k \lambda_i = 1$ and $r \neq a^i \in \cl \conv (\Omega')$. From the definition of $\Omega'$, we also have $h a^i = 1, a^i \in \cl \cone(\Omega)$. According to the extreme ray assumption of $r$, while it can be written as the conical combination (convex combination is also conical combination) of other points in $\cl \cone(\Omega)$, we know there exists $\gamma_i > 0$ such that $a^i = \gamma_i r$. Since $h  r = h  a^i = 1$, we have $\gamma_i = 1$, meaning $a^i = r$, which contradict to the assumption that $r \neq a^i$. 
So for any extreme ray $r \in \cl \cone(\Omega)$ with $h r = 1$, $r$ is an extreme point of $\cl \conv (\Omega')$. Since $\cl \conv (\Omega') = \cl \conv( \cl ({\Omega'}))$, 
so $r$ is an extreme point of $\cl \conv(\cl ({\Omega'}))$. By Lemma~\ref{lem: victor_klee_extreme}, we obtain that $r \in \cl ({\Omega'})$. By definition of $\Omega'$, it implies that either $r  \in (\Omega)_+$, or there exists different $\{r^i\} \subseteq \Omega$ such that $r^i \xrightarrow{c} r$. 
\end{proof}

\subsection{Proof of Lemma~\ref{lem: kronecker}}
\label{append: kroneck_lemma}

The next lemma says, the rational linear subspace $V_\pi$ defined in Definition~\ref{defn: subspace_pi} can be characterized by any linear basis of $\{1, \pi_1, \ldots, \pi_n\}$ over $\Q$. Let $e^1, \ldots, e^n$ denote the canonical basis of $\Z^n$. 
\begin{lemma}
\label{lem: V_pi_basis_independent}
Let $\{1, \pi_i \text{ for } i \in I\}$ be a linear basis of $\{1, \pi_1, \ldots, \pi_n\}$ over $\Q$, with $\pi_j = q_{j,0} + \sum_{i \in I} q_{j,i} \pi_i \ \forall j \notin I$, here $q_{j,i} \in \Q \ \forall i \in \{0\} \cup I, j \notin I.$
Then $V_\pi = \{x \in \R^n \mid x_j = \sum_{i \in I} q_{j,i} x_i \  \forall j \notin I\}$.
\end{lemma}

\begin{proof}
Denote $L: = \{x \in \R^n \mid x_j = \sum_{i \in I} q_{j,i} x_i, j \notin I\}$.
First, we show that $V_\pi \supseteq L$. For any $\alpha \in \Q^n$ such that $\alpha^T \pi \in \Q$, since $\pi_j = q_{j,0} + \sum_{i \in I} q_{j,i} \pi_i$ for any $j \notin I$, we have:
$\alpha^T \pi  = \sum_{i \in I} \alpha_i \pi_i + \sum_{j \notin I} \alpha_j (q_{j,0} + \sum_{i \in I} q_{j,i} \pi_i)  = \sum_{j \notin I} \alpha_j q_{j,0} + \sum_{i \in I} (\alpha_i + \sum_{j \notin I} \alpha_j q_{j,i}) \pi_i \in \Q.$
Since $\alpha \in \Q^n, q_{j,i} \in \Q$, and $\{1, \pi_i \text{ for } i \in I\}$ are linearly independent over $\Q$, hence we obtain that $\alpha_i + \sum_{j \notin I} \alpha_j q_{j,i} = 0$ for any $i \in I$. For any $x \in L$, by definition of $L$, we have
$\alpha^T x = \sum_{i \in I} (\alpha_i + \sum_{j \notin I} \alpha_j q_{j,i}) x_i$, which is simply 0. Therefore, we have shown $V_\pi \supseteq L$. 
Next, we show $L \supseteq V_\pi$. 
For any $j \notin I$, define $\alpha^j: = e^j - \sum_{i \in I}q_{j,i} e^i$.
Then easy to verify that, $(\alpha^j)^T \pi \in \Q$. So for any $x \in V_\pi$ and $j \notin I$, there is $(\alpha^j)^T x = 0$. This is simply saying, for any $j \notin I, x_j = \sum_{i \in I} q_{j,i} x_i$, which implies that $L \supseteq V_\pi$.
 \end{proof}

\begin{proof}[Proof of Lemma~\ref{lem: kronecker}]
W.l.o.g. we assume that a linear basis of $\{1, \pi_1, \ldots, \pi_n\}$ over $\Q$ is $\{1, \pi_1, \ldots, \pi_k\}$. 
Let $\pi_j = q_{j,0} + \sum_{i=1}^k q_{j,i} \pi_i$ for any $j >k$, here $q_{j,i} \in \Q \ \forall i \in \{0\} \cup [k], j > k.$
By Lemma~\ref{lem: V_pi_basis_independent}, we know $V_\pi = \{x \in \R^n \mid x_j = \sum_{i=1}^k q_{j,i} x_i, j > k\}$. When $k=n$ then $V_\pi = \R^n$, and the statement of this lemma is a special case of Weyl's criterion. We reduce the general case to this one.

The following elements lie in $\{x \in \R^n \mid x_j = \sum_{i=1}^k q_{j,i} x_i, j > k\}$, which is $V_\pi$:
$$
\tilde{e}^i := e^i + \sum_{j=k+1}^n q_{j,i} e^j \ \forall i \leq k, \quad\quad \tilde{\pi} = \pi - \sum_{j=k+1}^n q_{j,0} e^j. 
$$
By Weyl's criterion, $\Z^k + (\pi_1, \ldots, \pi_k) \Z_{> N_0}$ is dense in $\R^k$. We reformulate this for $V$ via the projection to the first $k$ coordinates, which is an isomorphism between $V$ and $\R^k$: a dense subset of $V_\pi$ is $\sum_{i=1}^k \Z \tilde{e}^i + \tilde \pi \Z_{>N_0}$, which is a subset of $\Z^n + \pi \Z_{> N_0}$. This completes the proof.
\end{proof}

\end{document}